\documentclass[12pt]{amsart}

\usepackage{epsfig,amssymb}

\makeatletter
\def\eqalign#1{\null\vcenter{\def\\{\cr}\openup\jot\m@th
  \ialign{\strut$\displaystyle{##}$\hfil&$\displaystyle{{}##}$\hfil
      \crcr#1\crcr}}\,}
\makeatother

\newcommand{\be}{\begin{equation}} 
\newcommand{\ee}{\end{equation}}
\newcommand{\beq}{\begin{eqnarray}}
\newcommand{\eeq}{\end{eqnarray}}
\newcommand{\bt}{\beta}
\newcommand{\et}{\end{theorem}}
\newcommand{\bl}{\begin{lemma}}
\newcommand{\el}{\end{lemma}}
\newcommand{\bc}{\begin{corollary}}
\newcommand{\ec}{\end{corollary}}
\newcommand{\ba}{\begin{array}}
\newcommand{\ea}{\end{array}}

\newcommand{\la}{\label}
\newcommand{\ci}{\cite}

\newcommand{\de}{\delta}
\newcommand{\De}{\Delta}
\newcommand{\al}{\alpha}
\newcommand{\ga}{\gamma}
\newcommand{\Ga}{\Gamma}

\newcommand{\si}{\sigma}
\newcommand{\Si}{\Sigma}
\newcommand{\om}{\omega}
\newcommand{\Om}{\Omega}
\newcommand{\lb}{\lambda}
\newcommand{\ze}{\zeta}
\newcommand{\ka}{\varkappa}

\newcommand{\ep}{\varepsilon }

\newcommand{\bi}{\bibitem}

\newfont{\msbm}{msbm10 scaled\magstep1}
\newfont{\msbms}{msbm7 scaled\magstep1} 

\newcommand{\bbr}{\mbox{$\mbox{\msbm R}$}}

\newcommand{\bbc}{\mbox{$\mbox{\msbm C}$}}

\newtheorem{theorem}{Theorem}[section]
\newtheorem{proposition}[theorem]{Proposition}

\theoremstyle{definition}

\theoremstyle{remark}
\newtheorem{remark}[theorem]{Remark}

\numberwithin{equation}{section}

\begin{document}
\def\wt{\widetilde}
\title[Gaussian weight with a jump]{Hankel determinant and orthogonal polynomials
for the Gaussian weight with a jump}
\author{A. Its}
\address{Department of Mathematical Sciences,
Indiana University -- Purdue University  Indianapolis,
Indianapolis, IN 46202-3216, USA}
\thanks{The first author was supported in part by NSF Grant \#DMS-0401009.}

\author{I. Krasovsky}
\address{Department of Mathematical Sciences,
Brunel University West London,
Uxbridge UB8 3PH, United Kingdom}

\thanks{The second author was supported in part by EPSRC Grant EP/E022928/1.}


\dedicatory{Dedicated to Percy Deift on the occasion of his 60th birthday.}

\begin{abstract}
We obtain asymptotics in $n$ for the $n$-dimensional Hankel determinant
whose symbol is the Gaussian multiplied by a step-like function. We use
Riemann-Hilbert analysis of the related system of orthogonal 
polynomials to obtain our results. 
\end{abstract}

\maketitle

\section{Introduction}
Consider the $n$-dimensional Hankel determinant
\be
D_n(\beta)=\det\left(\int_{-\infty}^\infty x^{j+k}
  w(x)dx\right)_{j,k=0}^{n-1}=
{1\over n!}\int_{-\infty}^\infty\cdots\int_{-\infty}^\infty
\prod_{i<j}{(x_i-x_j)^2}\prod_{k=1}^n{w(x_k)dx_k}
\ee
with discontinuous symbol
\be\la{w}
w(x)=e^{-x^2}\begin{cases}e^{i\beta\pi},& x < \mu_0\cr
e^{-i\beta\pi},& x\ge\mu_0\end{cases},\qquad \Re\bt\in (-1/2,1/2). 
\ee
We are interested in the asymptotics of $D_n(\beta)$ for large $n$.
Note that if the jump in $w(x)$ is absent, i.e. $\beta=0$,
$D_n(0)$ is a Selberg integral with the following explicit 
representation and the asymptotics (cf. \cite{Warc}):
\be\label{D0}
D_n(0)=(2\pi)^{n/2}2^{-n^2/2}\prod_{j=1}^{n-1}{j!}=
(2\pi)^n(n/2)^{n^2/2}n^{-1/12}e^{-(3/4)n^2+\ze'(-1)}(1+O({1/n})),
\ee
where $\ze'(x)$ is the derivative of Riemann's zeta-function.

The jump discontinuity similar to the one in (\ref{w}) and root-type singularities 
of the form 
$|x~-~\mu|^{\al}$ are collectively known as Fisher-Hartwig
singularities. 
 
Investigation of asymptotic behaviour of general {\it Toeplitz}
determinants whose symbol has such singularities was initiated by 
Lenard \cite{L} and 
Fisher and Hartwig \cite{FH}, who conjectured asymptotic
formulas on the basis of explicitly known examples and Szeg\H o's theorem 
for smooth symbols. (Such determinants
are related, for example, to the Ising model and to random walks on a lattice.)
A proof of part of the 
Fisher-Hartwig conjectures was obtained by Widom in 1973 \cite{W}.
Since then many workers have contributed to proof/disproof of these
conjectures (see \cite{Ehr} for a review) using mostly methods from operator theory. 

In contrast, the question of asymptotics for Hankel determinants, especially
when the symbol has unbounded support, remains to a large extent
open. Studies for in-a-sense regular perturbations of Hankel symbols
were carried out in \cite{J, BCW, KVA, BC}.
The investigation of the ``singular'' Fisher-Hartwig case for symbols on $\bbr$ 
started with the particular situation of the 
Gaussian perturbed by root-type singularities as this
example is important for random matrix theory; namely, for symbols
\be
e^{-x^2}\prod_{j=1}^m|x-\mu_j|^{2\al_j}.\la{root}
\ee

It turns out that the proper singular case corresponds to $\mu_j$
being inside the support of the equilibrium measure, namely,
$\mu_j\in [-\sqrt{2n},\sqrt{2n}]$. 
The asymptotics of the Hankel determinant with the symbol
(\ref{root}) and $\mu_j=\lb_j\sqrt{2n}$, $\lb_j\in (-1,1)$ were found in 
\cite{BH,FF,G} for integer $\al_j$, and in \cite{Kr2}, for the general
case $\Re\al_j>-1/2$. 

In the present paper we consider another (jump) type of Fisher-Hartwig
singularity on the Gaussian background. We prove

\begin{theorem}\la{Th1} 
Fix $\lb_0\in(-1,1)$, $\Re\beta\in(-1/4,1/4)$.
Let $\mu_0=\lb_0\sqrt{2n}$.
Then, as $n\to\infty$,
\be\eqalign{
\frac{D_n(\beta)}{D_n(0)}=
G(1+\beta)G(1-\beta)(1-\lb_0^2)^{-3\beta^2/2}
(8n)^{-\beta^2}\times\\
\exp\left\{2in\beta\left(\arcsin\lb_0+
\lb_0\sqrt{1-\lb_0^2}\right)\right\}\left[1+O\left({\ln n\over n^{1-4|\Re\beta|}}
\right)\right]
,}\label{as}
\ee
where  
$G(z)$ is Barnes' $G$-function.
\end{theorem}

With increasing effort, one can compute higher order asymptotics. The restriction $\Re\beta\in(-1/4,1/4)$ only makes
it easier to handle some technicalities.
One could use our methods to consider the general case $\Re\beta\in(-1/2,1/2]$
as well.

Note that 
\be\la{C}
G(1+\beta)G(1-\beta)=\left({\Gamma(\bt)\over\Gamma(-\bt)}
\right)^\bt \exp\left(-\bt^2-\int_0^\bt\ln {\Gamma(\bt)\over\Gamma(-\bt)}
d\bt\right),
\ee
where $\Gamma(z)$ is Euler's $\Gamma$-function.

It is interesting to compare Theorem \ref{Th1} with the result of \cite{B,BS}
on a Toeplitz determinant for a symbol with a jump. We have the same
combination of $G$-functions in both cases.

Of independent interest is the investigation of the system of polynomials 
orthonormal w.r.t. the weight (\ref{w}). Consider the system 
$p_k(x)=\ka_k x^k+\cdots$, $\ka_k\neq 0$,
$k=0,1,\dots$, satisfying 
\be\la{OP}
\int_{-\infty}^\infty p_j(x)p_k(x)w(x)dx=\de_{j\,k},\qquad k=0,1,\dots
\ee
This orthogonality relation is equivalent to
\be\la{OPp}
\int_{-\infty}^\infty p_k(x)x^j w(x)dx={\de_{jk}\over\ka_k},\qquad 
j=0,1,\dots,k,\quad k=0,1,\dots
\ee
We also consider monic polynomials $\hat p_k(x)=p_k(x)/\ka_k=x^k+\cdots$, 
which can
be defined by the conditions
\be\la{OPhp}
\int_{-\infty}^\infty\hat p_k(x)x^j w(x)dx=h_k\de_{jk},\quad h_k\neq 0,\qquad
j=0,1,\dots,k,\quad k=0,1,\dots
\ee
Obviously, $\ka_k^{-2}=h_k$.

The existence of such a system of orthogonal polynomials with nonzero 
normalization coefficients $h_k$ is well known for real weights. 
If $\Re\bt\neq 0$,
we will show existence of these polynomials for large enough degree.

As any system of orthogonal polynomials,  $\hat p_n(x)$ satisfy a three-term 
recurrence relation:
\be\la{rr}
x\hat p_n(x)=\hat p_{n+1}(z)+A_n \hat p_n(x)+B_{n} \hat p_{n-1}(x),\qquad
n=0,1,\dots,\qquad \hat p_{-1}\equiv 0.
\ee

Multiplying (\ref{rr}) by $\hat p_{n-1}(x)w(x)$ and integrating gives
\be\la{B}
B_n=h_{n-1}^{-1}\int_{-\infty}^\infty \hat p_n(x) \hat p_{n-1}(x)xw(x)dx=
{h_n\over h_{n-1}}={\ka^2_{n-1}\over\ka^2_n}.
\ee

Let us fix the notation for the two leading coefficients of the polynomials
$\hat p_n(x)$ as follows:
\[
\hat p_n(x)=x^n+\beta_n x^{n-1}+\ga_n x^{n-2}+\cdots
\]
Comparing the coefficients at the powers $x^n$ and
$x^{n-1}$ in the recurrence relation (\ref{rr}), we obtain 
the following identities we use later on:
\be\la{coeffid}
A_n=\beta_n-\beta_{n+1},\qquad
\left(\ka_{n-1}\over \ka_n\right)^2=\ga_n-\ga_{n+1}
-\beta_n^2+\beta_n\beta_{n+1}.
\ee

Our methods allow us to compute various asymptotics for polynomials $p_n(x)$
as $n\to\infty$.
One such question is how the singularity of the
weight affects the recurrence coefficients $A_n$, $B_n$ for large $n$.
This was discussed by Chen and Pruessner \cite{CP} who found 
a behavior oscillating in $n$ involving $\cos(c_1\ln n+ c_2)$,
but did not determine the constant phase $c_2$.
We prove

\begin{theorem}\la{Th2}
Fix $\lb_0\in(-1,1)$, $\Re\beta\in (-1/2,1/2)$, $\beta\neq 0$.
Let $\mu_0=\lb_0\sqrt{2n}$.
Then as $n\to\infty$, the recurrence coefficients satisfy
\be\la{An}
\eqalign{
A_n=-{i\bt^2\sin\pi\bt\over 4\pi\sqrt{2n}}
\left(
(a+a^{-1})(8n)^\bt (1-\lb_0^2)^{3\bt/2}\Gamma(-\bt)
e^{n\phi_+(\lb_0)}\right.\\
\left.
+i(a-a^{-1})(8n)^{-\bt}(1-\lb_0^2)^{-3\bt/2}\Gamma(\bt)
e^{-n\phi_+(\lb_0)}
\right)^2\left[1+O\left({1\over n^{1-2|\Re\beta|}}\right)\right],}
\ee
\be\la{Bn}
\eqalign{
B_n={n\over 2}-{i\bt\lb_0\over 2\sqrt{1-\lb_0^2}}+
{\pi\sinh\phi_+(\lb_0)\over 4(1-\lb_0^2)\sin\pi\bt}\\
\times
\left[(8n)^{2\bt} (1-\lb_0^2)^{3\bt}\Gamma(\bt)^{-2}
e^{(2n+1)\phi_+(\lb_0)+i\arcsin\lb_0}\right.\\
\left.
+(8n)^{-2\bt} (1-\lb_0^2)^{-3\bt}\Gamma(-\bt)^{-2}
e^{-(2n+1)\phi_+(\lb_0)-i\arcsin\lb_0}\right]+
O\left({1\over n^{1-2|\Re\beta|}}\right),}
\ee
where 
\[
a(z)=\left({1-\lb_0\over 1+\lb_0}\right)^{1/4}e^{i\pi/4},\qquad 
\phi_+(\lb_0)=2i\int_{\lb_0}^1\sqrt{1-x^2}dx.
\]

In particular, for $\bt=i\ga$, $\ga\in\bbr\setminus\{0\}$, 
\be\la{An2} 
\eqalign{
A_n={2\ga\over\sqrt{2n(1-\lb_0^2)}}\sin^2\left(n\left(\arcsin\lb_0+
\lb_0\sqrt{1-\lb_0^2}-{\pi\over2}\right)
-\ga\ln(8n)\right.\\
\left.
+\arg\Gamma(i\ga)-{3\ga\over 2}\ln(1-\lb_0^2)+\arctan{\lb_0\over 
1+\sqrt{1-\lb_0^2}}\right)
[1+O(1/n)],}
\ee
\be\la{Bn2}
\eqalign{
B_n={n\over 2}+ {\ga\lb_0\over 2\sqrt{1-\lb_0^2}}+
{\ga\over 2(1-\lb_0^2)}\cos\left(\arcsin\lb_0+
\lb_0\sqrt{1-\lb_0^2}\right)\\
\times\cos\left((2n+1)\left(\arcsin\lb_0+
\lb_0\sqrt{1-\lb_0^2}-{\pi\over 2}\right)
-2\ga\ln(8n)\right.\\
\left.
+2\arg\Gamma(i\ga)-3\ga\ln(1-\lb_0^2)-\arcsin\lb_0\right)
+O(1/n).}
\ee

If furthermore 
$\lb_0=0$, the above equations become
\begin{eqnarray}
A_n={2\ga\over\sqrt{2n}}\sin^2(\ga\ln(8n)-\arg\Gamma(i\ga)+\pi n/2)
[1+O(1/n)],\la{An3}\\
B_n={n\over 2}+ {\ga\over 2}
\cos(2\ga\ln(8n)-2\arg\Gamma(i\ga)+\pi(n+1/2))
+O(1/n).\la{Bn3}
\end{eqnarray}
\end{theorem}

Except for the phase
$\arg\Gamma(i\ga)$,
(\ref{An3}) and (\ref{Bn3}) were found in \cite{CP}. 
Note also that similar formulas for the recurrence
coefficients for the {\it Jacobi} weight perturbed by a jump were conjectured
by Magnus \cite{Magnus}. Our approach can be modified to verify
this conjecture.\footnote{In fact, 
Magnus conjectured asymptotics of the recurrence coefficients
for a perturbation which is a jump {\it multiplied} by a root-like singularity
at the same point.
Our approach would work in this general case as well. The case of
the Jacobi weight perturbed by only root-like singularities was resolved 
by Vanlessen in \cite{V}.}
We will address this question in a future publication.

Our proofs of Theorems \ref{Th1} and \ref{Th2} are based on 
the Riemann-Hilbert problem
(RHP) approach to orthogonal polynomials 
\cite{FIK} (see the next section), combined with the steepest
descent method for RHP of Deift and Zhou \cite{DZ} which proved to be very
successful in the analysis of asymptotic behavior of orthogonal
polynomials and related determinants (see, \cite{Dbook} for
an introduction and bibliography of earlier works, and 
\cite{KM,KV,V,KVA,BK,D,JBD,Kr1,Kr2,DIKZ,DIK2}).
The  steepest
descent method for a RHP for the orthogonal polynomials given by 
expression (\ref{OP})
allows us to prove Theorem \ref{Th2}. The main technical difficulty
here is the construction of an approximate solution
(parametrix) of a related RHP
in the neighborhood of $\lb_0$. We find a new parametrix given in
terms of the confluent hypergeometric function.

In order to prove Theorem \ref{Th1}, an additional analysis is needed. It is
based on a classical formula connecting $D_n(\beta)$ and the
orthogonal polynomials $p_k(z)=\ka_k z^k+\cdots$, namely  
\be
D_n(\beta)=\prod_{j=0}^{n-1}\ka_j^{-2}.\la{Dchi}
\ee
From this expression (cf. \cite{EM, BI, Kr1, Kr2, DIK2}) one 
can derive an identity for $(d/d\beta)\ln D_n(\beta)$  in terms of only $p_n(\lb_0)$, $p_{n-1}(\lb_0)$,
and several leading coefficients of $p_n(x)$.
Substituting the asymptotic expansions for the polynomials in this
identity and integrating over $\beta$ from $\beta=0$ to some $\beta$,
we obtain the asymptotics of $D_n(\beta)/D_n(0)$.

\section{Riemann-Hilbert problem for $p_k(z)$}

Consider the following Riemann-Hilbert problem for a 
$2\times 2$ matrix valued function $Y(z,k)\equiv Y(z)$ and the weight
$w(x)$ given in (\ref{w}):
\begin{enumerate}
    \item[(a)]
        $Y(z)$ is  analytic for $z\in\bbc \setminus\bbr$.
    \item[(b)] 
Let $x\in\bbr\setminus\mu_0$.
$Y$ has $L_2(\bbr)$ boundary values
$Y_{+}(x)$ as $z$ approaches $x$ from
above, and $Y_{-}(x)$, from below, 
related by the jump condition
\begin{equation}\label{RHPYb}
            Y_+(x) = Y_-(x)
            \begin{pmatrix}
                1 & w(x) \cr
                0 & 1
             \end{pmatrix},
            \qquad\mbox{$x\in\bbr\setminus\mu_0$.}
        \end{equation}
    \item[(c)]
        $Y(z)$ has the following asymptotic behavior at infinity:
        \begin{equation} \label{RHPYc}
            Y(z) = \left(I+ O \left( \frac{1}{z} \right)\right)
            \begin{pmatrix}
                z^{k} & 0 \cr
                0 & z^{-k}\end{pmatrix}, \qquad \mbox{as $z\to\infty$.}
        \end{equation}
\end{enumerate}

It is easy to verify that, if the polynomials satisfying (\ref{OPhp}) exist,
this problem has a solution given by the function:
\begin{equation} \label{RHM}
    Y(z) =
    \begin{pmatrix}
\ka_k^{-1}p_k(z) & 
\ka_k^{-1}\int_{-\infty}^{\infty}{p_k(\xi)\over \xi-z}
{w(\xi)d\xi \over 2\pi i } \cr
-2\pi i\ka_{k-1}p_{k-1}(z) & 
-\ka_{k-1}\int_{-\infty}^{\infty}{p_{k-1}(\xi)\over \xi-z} w(\xi)d\xi 
    \end{pmatrix}.
\end{equation}

Note, in particular, that $\det Y(z)=1$. Indeed, from the conditions on $Y(z)$, 
$\det Y(z)$ is analytic across the real axis, has all singularities 
removable, and tends to $1$ as $z\to\infty$. 
It is then identically $1$ by Liouville theorem. Also, the solution is unique: 
if there is another solution $\wt Y(z)$, we easily obtain by Liouville theorem
that $Y(z) \wt Y(z)^{-1}\equiv 1$.

In the next section we derive an expression for 
$(d/d\beta)\ln D_n(\beta)$
in terms of the matrix elements of (\ref{RHM}), and 
in the section after that we compute the asymptotics of $Y(z)$ using (a) -- (c).

The existence of the system of orthogonal polynomials $p_k(z)=\ka_k z^k+\cdots$ 
satisfying (\ref{OP})
with nonzero leading coefficients $\ka_k$ for real $w(x)$ is a 
classical fact. Moreover, the coefficients $\ka_k^2=D_k/D_{k+1}$ 
are regular functions of $\bt$
(as follows from the determinantal representation for $D_k$ \ci{Sz, Dbook}). 
For all complex $\bt$ 
in any fixed closed bounded set of the strip $-1/2<\Re\bt<1/2$, 
(denote this set by $\wt\Om$),
we shall prove below the existence of a solution to
the Riemann-Hilbert problem for all $n$ larger then some $n_0>0$. Its asymptotics
will be explicitly 
constructed. We shall also see that the coefficients $\ka_k$, $k>n_0$ are nonzero 
and finite for all such $\bt$. For $k\le n_0$, the coefficients $\ka_k^2$ 
are regular
and nonzero (as follows from the determinantal representation) 
outside of a possible subset $\Om$ of $\wt\Om$. 
As a consequence of the 
determinantal representation, the system of the orthogonal polynomials exists, and
the formula (\ref{Dchi}) holds for $\bt\in \wt\Om\setminus \Om$. 
Throughout Section 3 
(and, hence, in the differential identity 
obtained there) we assume that $\bt\notin\Om$. 
(There is no such condition on $\bt$ in Section 4.) This provides in Section 5
a proof of Theorem \ref{Th1} for $\bt$ outside the set $\Om$. 
However, as we shall see
below, the error term in the asymptotics of $D_n$ is uniform for {\it all} 
$\bt\in\wt\Om$.
Theorem \ref{Th1} will follow then by continuity.

\section{Differential identity}
We now derive a differential identity needed to prove Theorem \ref{Th1}.
Throughout this section, we consider $n$ a {\it fixed} positive integer and
$\bt\in \wt\Om\setminus \Om$ (see previous section).
We start with the following identity obtained in (\cite{Kr2}, eq. (17)) 
for a weight $w(x)$ depending on a parameter $\bt$:
\be\la{D2}\eqalign{
{d\over d\bt}\ln D_n(\bt)=
-n{\ka'_{n-1,\bt}\over\ka_{n-1}}+{\ka_{n-1}\over\ka_n}(J_1-J_2),\\
J_1=\int_{-\infty}^\infty p'_{n,\bt}(x)p'_{n-1,x}(x)w(x)dx,
\qquad
J_2=\int_{-\infty}^\infty p'_{n,x}(x)p'_{n-1,\bt}(x)w(x)dx.}
\ee
Here the prime and the lower index $\bt$ (respectively, $x$) 
stand for the derivative w.r.t. 
$\bt$ (respectively, $x$).

This identity contains polynomials of order $n$ and $n-1$ only and 
therefore can be used to calculate asymptotics of Hankel determinants
for rather general weights $w(x)$. However, as in the situation
for weights with root-type singularities considered in \cite{Kr2},
the weights with jumps, as will be shown below, allow us to 
reduce the above
identity to a ``local'' one involving only the values of the matrix elements of
$Y(z,n)$ at particular points. More precisely, for our weight (\ref{w})
we will obtain an identity in terms of $p_n(\mu_0)$, $p_{n-1}(\mu_0)$,
their derivatives, and several polynomial coefficients. 

Note that $p_n(x)$ are analytic functions of $\bt$, $\bt\notin\Om$,
as follows, e.g., from their representation as a determinant.

For the reader's convenience,
we recall first the derivation of (\ref{D2}).
Using (\ref{Dchi}) and the orthogonality condition (\ref{OPp}), we have
(note that $p'_{j,\bt}(x)=\ka'_{j,\bt}x^j+\mbox{polynomial of degree less than j}$)
\be\la{D1}\eqalign{
{d\over d\bt}\ln D_n(\bt)={d\over d\bt}\ln 
\prod_{j=0}^{n-1}\ka_j^{-2}=
-2\sum_{j=0}^{n-1}{\ka'_{j,\bt}\over \ka_j}=
-2\sum_{j=0}^{n-1}
\int_{-\infty}^\infty p_j(x)p'_{j,\bt}(x)w(x)dx=\\
-\int_{-\infty}^\infty \left(\sum_{j=0}^{n-1}p^2_j(x)\right)'_\bt
w(x)dx.}
\ee

By the Christoffel-Darboux formula (e.g., \ci{Sz}),
\be\la{CD}
\sum_{j=0}^{n-1}p^2_j(x)={\ka_{n-1}\over \ka_n}
(p'_{n,x}(x)p_{n-1}(x)-p_n(x)p'_{n-1,x}(x)).
\ee

Substituting (\ref{CD}) into (\ref{D1}), differentiating it
w.r.t. $\bt$, and using orthogonality, we obtain (\ref{D2}).

Let us now evaluate $J_1$. 
Integrating by parts, 
\be\eqalign{
J_1= \int_{-\infty}^{\mu_0} p'_{n,\bt}(x)p'_{n-1,x}(x)e^{-x^2}e^{i\pi\bt}dx
+
\int_{\mu_0}^\infty p'_{n,\bt}(x)p'_{n-1,x}(x)e^{-x^2}e^{-i\pi\bt}dx\\
=2\sinh{(i\pi\bt)}e^{-\mu_0^2}p'_{n,\bt}(\mu_0)p_{n-1}(\mu_0)-
\int_{-\infty}^\infty p_{n-1}(x)p''_{n,\bt,x}(x)w(x)dx\\
+2\int_{-\infty}^\infty p_{n-1}(x)p'_{n,\bt}(x) x w(x)dx.}
\ee

The first integral in the r.h.s. equals $n\ka'_{n,\bt}/\ka_{n-1}$ by 
orthogonality, and the second can be written as
\[
B\equiv\int_{-\infty}^\infty
p_{n-1}(x)(\ka'_{n,\bt}x^{n+1}+(\ka_n\beta_n)'_\bt
x^n+(\ka_n\ga_n)'_\bt x^{n-1}+\cdots)w(x)dx.
\]
As in \cite{Kr2},
expanding $x^{n+1}$ and $x^n$ in terms of the polynomials $p_j(x)$,
we obtain, by orthogonality and (\ref{coeffid}),
\be
B={\ka_{n}\over \ka_{n-1}}\left[
{\ka'_{n,\bt}\over \ka_n}
\left(\ka_{n-1}\over \ka_n\right)^2
+\ga'_{n,\bt}-\beta_n\beta'_{n,\bt}\right].
\ee
Finally,
\be
J_1=-n{\ka'_{n,\bt}\over \ka_{n-1}}+
2\sinh(i\pi\bt)e^{-\mu_0^2}p'_{n,\bt}(\mu_0)p_{n-1}(\mu_0)+2B.
\ee
A similar calculation for $J_2$ yields
\be
J_2=2{\ka'_{n-1,\bt}\over \ka_n}+
2\sinh(i\pi\bt)e^{-\mu_0^2}p_n(\mu_0)p'_{n-1,\bt}(\mu_0).
\ee
Substituting the above expressions into (\ref{D2})
and using (\ref{RHM}), we obtain

\begin{proposition}\la{Prop1} 
Let $\bt$, $-1/2<\Re\bt\le 1/2$, be such that the
system of orthogonal polynomials $p_k(z)$ with finite 
non-vanishing leading coefficients 
satisfying (\ref{OP}) exists.
Fix $n>1$. Let $p_n=\ka_n(x^n+\beta_n x^{n-1}+\gamma_n
x^{n-2}+\cdots)$, and the matrix $Y(z)$ be given by (\ref{RHM}). 
Then   
\be\eqalign{
{d\over d\bt}\ln D_n(\bt)=
-n(\ln\ka_n\ka_{n-1})'_{\bt}-
\left(\ka_{n-1}^2\over \ka_n^2\right)'_\bt+
2\left(\ga'_{n,\bt}-\beta_n\beta'_{n,\bt}\right)+\\
{1\over\pi}e^{-\mu_0^2}\sin{(\pi\bt)}
\left[Y_{11}(\mu_0)Y_{21}(\mu_0)'_\bt-
Y_{11}(\mu_0)'_{\bt} Y_{21}(\mu_0)
-(\ln\ka_n\ka_{n-1})'_{\bt}Y_{11}(\mu_0) Y_{21}(\mu_0)\right].}\la{id}
\ee
\end{proposition}


\section{Asymptotic analysis of the Riemann-Hilbert problem}

The standard steps of the steepest descent analysis 
applied to the Riemann-Hilbert problem (a) -- (c) of Section 2 
involve rescaling by $\sqrt{2n}$ so that the support of 
the equilibrium measure becomes $[-1,1]$, deformation of the contour
such that the jump matrix becomes close to the identity,
construction of local parametrices around the points $1$, $-1$, and 
$\lb_0=\mu_0/\sqrt{2n}$,
and matching the parametrices with the solution in the region outside 
those points. All these steps are standard except for the construction
of the parametrix around $\lb_0$. 
In this section we carry out the analysis and obtain Theorem \ref{Th2} as well as
the asymptotics necessary for the proof of Theorem \ref{Th1} in the last section.
Henceforth we assume that $\lb_0\in (-1,1)$.

\subsection{Three transformations of the Riemann-Hilbert problem} 
We now perform standard transformations of the RHP.
The first one $Y\to U$ is a scaling:
\be\la{U}
Y(z\sqrt{2n})=(2n)^{n\si_3/2}U(z),\qquad \si_3=\begin{pmatrix}1&0\cr 0&-1
\end{pmatrix}.
\ee  
The second one $U\to T$ is given by the formula
\be\la{T}
U(z)=e^{nl\si_3/2}T(z)e^{n(g(z)-l/2)\si_3},
\ee
where
\be\la{gf}\eqalign{
l=-1-2\ln 2,\\
g(z)=\int_{-1}^1\ln(z-s)\psi(s)ds,\qquad z\in\bbc\setminus(-\infty,1],\qquad
\psi(z)={2\over\pi}\sqrt{1-z^2}.}
\ee
Below we
always take the principal branch of the logarithm and roots.
The function $\psi(z)$
is the scaled asymptotic density of zeros of $p_n(z)$. 
(Note that $\int_{-1}^1\psi(x)dx=1$.)
The function $g(z)$ has the following useful properties:
\be\la{g}
\eqalign{
g_+(x)+g_-(x)-2x^2-l=0,\qquad \mathrm{for}\quad x\in(-1,1)\\
g_+(x)+g_-(x)-2x^2-l<0,\qquad \mathrm{for}\quad x\in\bbr\setminus[-1,1]\\
g_+(x)-g_-(x)=\begin{cases}2\pi i, & \mbox{for}\quad x\le-1\cr
2\pi i\int_x^1\psi(y)dy, & \mbox{for}\quad x\in [-1,1]\cr
0, & \mbox{for}\quad x\ge 1\end{cases}}
\ee

From the RHP for $Y(z)$, we obtain the following 
problem for $T(z)$:

\begin{enumerate}
    \item[(a)]
        $T(z)$ is  analytic for $z\in\bbc \setminus\bbr$.
    \item[(b)]  
The boundary values of $T(z)$ are related by the jump condition
\begin{equation}
\eqalign{   T_+(x) = T_-(x)
            \begin{pmatrix}
                e^{-n(g_+(x)-g_-(x))} & \om(x) \cr
                0 & e^{n(g_+(x)-g_-(x))}\end{pmatrix},
            \qquad\mbox{$x\in(-1,1)$.}\\
            T_+(x) = T_-(x)
            \begin{pmatrix}
                1 & \om(x)e^{n(g_+(x)+g_-(x)-2x^2-l)}\cr
                0 & 1\end{pmatrix},
            \qquad\mbox{$x\in\bbr\setminus[-1,1]$.}}
        \end{equation}
    \item[(c)]
$T(z)=I+O(1/z)$ as $z\to\infty$,
\end{enumerate}
where $\om(x)$ are the values on the real axis of the following 
function defined in the whole complex plane with a slit along the line
$\Re z=\lb_0$:
\be
\om(z)=\begin{cases}e^{i\beta\pi},& \Re z<\lb_0\cr
e^{-i\beta\pi},& \Re z\ge\lb_0\end{cases}.
\ee
Note that this problem is normalized to $1$ at infinity, and 
the jump matrix on $(-\infty,-1)\cup (1,\infty)$ is exponentially close
to the identity (see (\ref{g})). 
In order to have uniform bounds,
we have to exclude small neighborhoods of the points $-1$ and $1$,
where $g_+(x)+g_-(x)-2x^2-l$ is close to zero.

Let $h(z)$ be the analytic continuation of 
\be
h(x)=g_+(x)-g_-(x)=2\pi e^{3i\pi/2}\int_1^x\psi(y)dy
\ee
to $\bbc\setminus((-\infty,-1]\cup[1,\infty))$.
A simple analysis shows that $\Re h(z)>0$ for $\Im z>0$,
and $\Re h(z)<0$ for $\Im z<0$ in some neighborhood of $(-1,1)$. 
We again exclude 
neighborhoods of the points $-1$ and $1$ for a uniform estimate.

By the steepest descent method of Deift and Zhou, we now split
the contour as shown in
Figure 1. 
\begin{figure}
\centerline{\psfig{file=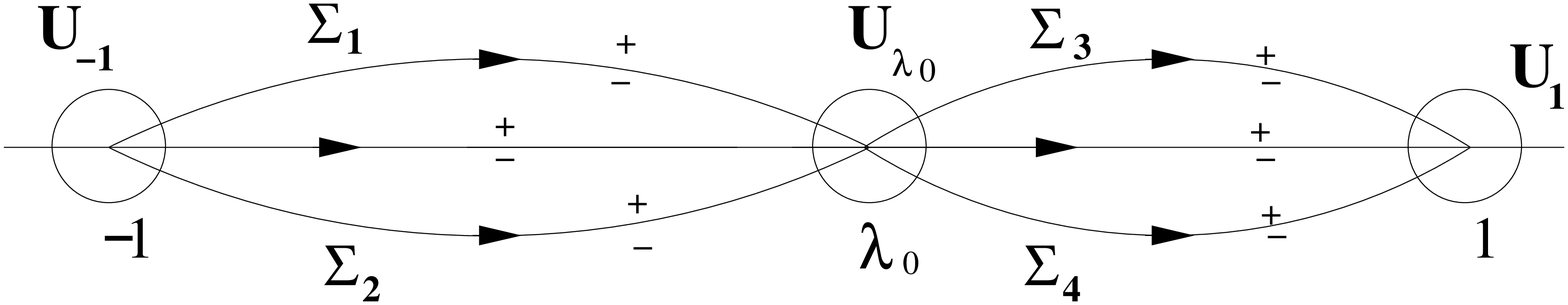,width=4.0in,angle=0}}
\vspace{0cm}
\caption{
Contour for the $S$-Riemann-Hilbert problem.}
\label{fig1}
\end{figure}
Define a new transformation of our matrix-valued function as follows:
\be\la{defS}
S(z)=
\begin{cases}T(z),& \mbox{for $z$ outside the lenses},\cr
T(z)\begin{pmatrix}1 & 0\cr -\om(z)^{-1}e^{-nh(z)}& 1\end{pmatrix}, & 
\mbox{for $z$ in the upper part of the lenses},\cr
 T(z)\begin{pmatrix}1 & 0\cr \om(z)^{-1}e^{nh(z)}& 1\end{pmatrix},&
\mbox{for $z$ in the lower part of the lenses}.
\end{cases}
\ee

Then the Riemann-Hilbert problem for $S(z)$ is the following:

\begin{enumerate}
    \item[(a)]
        $S(z)$ is  analytic for $z\in\bbc \setminus\Sigma$, where 
$\Si=\bbr\cup \cup_{j=1}^{4}\Si_j$.
    \item[(b)]  
The boundary values of $S(z)$ are related by the jump condition
\begin{equation}
\eqalign{   S_+(x) = S_-(x)
            \begin{pmatrix}
                  1 & 0\cr
                  \om(x)^{-1}e^{\mp n h(x)} & 1\end{pmatrix},
            \qquad\mbox{$x\in\cup_{j=1}^{4}\Sigma_j$},\\
\mbox{where the plus sign in the exponent is on 
$\Sigma_2$, $\Sigma_4$,
 and minus, on $\Sigma_1$, $\Sigma_3$,}\\
            S_+(x) = S_-(x)
              \begin{pmatrix}
                  0 & \om(x)\cr
                  -\om(x)^{-1} & 0\end{pmatrix},
            \qquad\mbox{$x\in(-1,1)$.}\\
            S_+(x) = S_-(x)
            \begin{pmatrix}
                1 & \om(x)e^{n(g_+(x)+g_-(x)-2x^2-l)}\cr
                0 & 1\end{pmatrix},
            \qquad\mbox{$x\in\bbr\setminus[-1,1]$.}}
        \end{equation}
    \item[(c)]
$S(z)=I+O(1/z)$ as $z\to\infty$.
\end{enumerate}

Recalling the remarks above, we see that,
outside the neighborhoods $U_{\lb_0}$, $U_{\pm1}$,
the jump matrix on $\Sigma_j$, $j=1,\dots,4$ is 
uniformly exponentially close to 
the identity. (Note that we had to contract the lenses at $\lb_0$, because
of the cut of $\om(z)$ across the line $\Re z=\lb_0$.)
We shall now construct the parametrices in $U_{\lb_0}$, $U_{\pm1}$,
and $U_\infty=\bbc\setminus(U_{\lb_0}\cup U_1\cup U_{-1})$. After that we shall
match them on the boundaries $\partial U_{\lb_0}$ and $\partial U_{\pm 1}$,
which will yield the desired asymptotics.

\subsection{Parametrix in $U_\infty$}

We expect the following problem for the parametrix $P_\infty$ in $U_\infty$:

\begin{enumerate}
    \item[(a)]
        $P_\infty(z)$ is  analytic for $z\in\bbc \setminus[-1,1]$,
    \item[(b)]
with the jump condition on $(-1,1)$
\be
P_{\infty,+}(x) = P_{\infty,-}(x)
            \begin{pmatrix}0&\om(x)\cr
              -\om(x)^{-1}&0\end{pmatrix},
\qquad\mbox{$x \in (-1,1)$},
\ee
\item[(c)]
and the following behavior at infinity
\be
P_\infty(z) = I+ O \left( \frac{1}{z} \right), 
     \qquad \mbox{as $z\to\infty$.}
\ee
\end{enumerate}
A solution $P_\infty(z)$ can be found in the same way as, e.g., in \ci{KVA}:
\be
P_\infty(z)={1\over 2}(\mathcal{D}_\infty)^{\si_3}
\begin{pmatrix}a+a^{-1}&-i(a-a^{-1})\cr i(a-a^{-1})& a+a^{-1}\end{pmatrix}
\mathcal{D}(z)^{-\si_3},
\qquad
a(z)=\left({z-1\over z+1}\right)^{1/4},\la{N}
\ee
where the cut of the root is the interval $(-1,1)$. 
Note that $\det P_\infty(z)=1$.
Here
\be
\mathcal{D}(z)=\exp\left[{\sqrt{z^2-1}\over 2\pi}\int_{-1}^1
{\ln \om(\xi)\over \sqrt{1-\xi^2}}
{d\xi \over z-\xi}\right],\qquad \mathcal{D}_\infty=\lim_{z\to\infty}
\mathcal{D}(z).\la{Sf}
\ee

The Szeg\H o function $\mathcal{D}(z)$ is
analytic outside the interval $[-1,1]$
with boundary values satisfying $\mathcal{D}_+(x)\mathcal{D}_-(x)=\om(x)$,
$x\in (-1,1)$.
Calculation of the elementary integral in (\ref{Sf}) gives
\be\la{calD}
\eqalign{\mathcal{D}(z)=
\left(z\lb_0-1-i\sqrt{(z^2-1)(1-\lb_0^2)}\over z-\lb_0\right)^\bt e^{i\pi\bt/2}\\
=\left(z-\lb_0\over z\lb_0-1+i\sqrt{(z^2-1)(1-\lb_0^2)}\right)^\bt e^{i\pi\bt/2}.}
\ee
From here it is easy to obtain as the main term of the expansion in $1/z$
at infinity
\be
\mathcal{D}_\infty=\left(i\lb_0+\sqrt{1-\lb_0^2}\right)^\bt=e^{i\bt\arcsin\lb_0}.
\ee
Below we will need 2 more terms in the expansion:
\be\la{calDinf}
{\mathcal{D}(z)\over \mathcal{D}_\infty}
=1-{i\bt\over z}\sqrt{1-\lb_0^2}-
{1\over 2z^2}\left(\bt^2(1-\lb_0^2)+i\bt\lb_0\sqrt{1-\lb_0^2}\right)
+O(z^{-3}),\qquad z\to\infty.
\ee
At the point $\lb_0$,
\be\label{Dlb}
\mathcal{D}(z)=c(\bt)(z-\lb_0)^\bt(1+o(1)),\quad c(\bt)=
2^{-\bt}(1-\lb_0^2)^{-\bt}e^{-i\pi\bt/2},\quad z\to\lb_0,\quad \Im z>0.
\ee
Note that in order to obtain the correct argument of the complex number $c(\bt)$, we returned to the analysis of the original integral in (\ref{Sf}).

Moreover, we will need the following expansions at $\pm 1$ which are 
easy to obtain from (\ref{calD}) and the definition of $a(z)$:
\be\la{calDp1}
\eqalign{
{\mathcal{D}^2(z)\over\om(z)}+
{\om(z)\over\mathcal{D}^2(z)}=
2\left( 1- 4\bt^2 {1+\lb_0 \over 1-\lb_0}u+O(u^{3/2})\right),\\
{\mathcal{D}^2(z)\over\om(z)}-
{\om(z)\over\mathcal{D}^2(z)}=
2^{5/2}i\bt\sqrt{1+\lb_0 \over 1-\lb_0}\sqrt{u}+O(u^{3/2}),\\
a^2+a^{-2}=\sqrt{2\over u}+{3\over 2}\sqrt{u\over 2}+O(u^{3/2}),
\qquad 
a^2-a^{-2}=-\sqrt{2\over u}+{1\over 2}\sqrt{u\over 2}+O(u^{3/2}),\\
u=z-1,\qquad z\to 1.}
\ee
and
\be\la{calDm1}
\eqalign{
{\mathcal{D}^2(z)\over\om(z)}+
{\om(z)\over\mathcal{D}^2(z)}=
2\left( 1+ 4\bt^2 {1-\lb_0 \over 1+\lb_0}u+O(u^{3/2})
\right),\\
{\mathcal{D}^2(z)\over\om(z)}-
{\om(z)\over\mathcal{D}^2(z)}=
-2^{5/2}\bt\sqrt{1-\lb_0 \over 1+\lb_0}\sqrt{u}+O(u^{3/2}),\\
a^2+a^{-2}=i\left(\sqrt{2\over u}-{3\over 2}\sqrt{u\over 2}\right)+O(u^{3/2}),
\quad 
a^2-a^{-2}=i\left(\sqrt{2\over u}+{1\over 2}\sqrt{u\over 2}\right)+O(u^{3/2}),\\
u=z+1,\qquad z\to -1.}
\ee



\subsection{Parametrix at the jump point}

Let us now construct the parametrix $P_{\lb_0}(z)$ in $U_{\lb_0}$. 
We look for an analytic matrix-valued function in a 
neighborhood of $U_{\lb_0}$
which satisfies the same jump conditions as $S(z)$ on
$\Sigma\cap U_{\lb_0}$, has the same behavior as $z\to\lb_0$, and 
satisfies the matching condition
\be\la{match}
P_{\lb_0}(z)P_\infty^{-1}(z)=I+o(1)
\ee
uniformly on the boundary $\partial U_{\lb_0}$ as $n\to\infty$.

Using the analytic continuation of $\psi(y)$ (see (\ref{gf})), define:
\be
\phi(z)=\begin{cases}h(z)/2=e^{3i\pi/2}\pi\int_1^z\psi(y)dy,& 
\mbox{for}\quad \Im z>0,\cr
e^{-i\pi}h(z)/2=e^{i\pi/2}\pi\int_1^z\psi(y)dy,& 
\mbox{for}\quad \Im z<0\end{cases}.
\ee
Clearly, $e^{\phi(z)}$ is analytic 
outside $[-1,1]$. Set
\be\la{f}
\hat f(z)=2\pi e^{i\pi/2}\int_{\lb_0}^z\psi(y)dy.
\ee
Let us now choose the exact form of the cuts $\Sigma$ in $U_{\lb_j}$ so that their images
under the mapping $\zeta=n\hat f(z)$ are straight lines (Figure 2).
Note that $\zeta(z)=n\hat f(z)$ is analytic and one-to-one in the neighborhood 
of $U_{\lb_0}$, and it takes the real axis to the imaginary axis.
We have
\be\la{zlb}
\zeta=n \hat f(z)=
4ne^{i\pi/2}\sqrt{1-\lb_0^2}(z-\lb_0)(1+O(z-\lb_0)),\qquad z\to\lb_0.
\ee

We look for $P_{\lb_0}(z)$ in the form
\be\la{Plb0}
P_{\lb_0}(z)=E(z)P^{(1)}(z)e^{-n\phi(z)\si_3},
\ee
where $E(z)$ is analytic and invertible in the neighborhood of $U_{\lb_0}$,
and therefore does not affect the jump and analyticity conditions. 
It is chosen so that the matching condition is satisfied.
It is easy to verify that $P^{(1)}(z)$ satisfies jump 
conditions with {\it constant}
jump matrices. Set
\be\la{P1}
P^{(1)}(z)=\Psi_{\bt}(\zeta)=\Psi_{\bt}(n \hat f(z)),
\ee
where $\Psi_{\bt}(\zeta)$ satisfies a RHP along cuts given in Figure 2: 

\begin{figure}
\centerline{\psfig{file=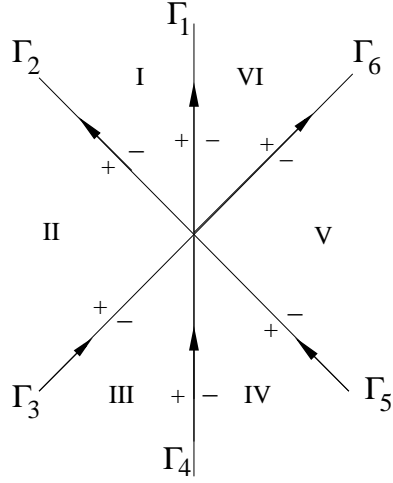,width=2.0in,angle=0}}
\vspace{0cm}
\caption{
The auxiliary contour for the parametrix at $\lb_0$.}
\label{fig2}
\end{figure}

\begin{enumerate}
\item[(a)]
    $\Psi_{\bt}$ is analytic for $\ze\in\bbc\setminus\cup_{j=1}^6\Ga_j$.
\item[(b)]
    $\Psi_{\bt}$ satisfies the following jump conditions:
    \begin{eqnarray} 
        \Psi_{\bt,+}(\zeta)
        &=&
            \Psi_{\bt,-}(\zeta)
            \begin{pmatrix}
                0 & e^{-i\pi\bt} \cr
                -e^{i\pi\bt} & 0
            \end{pmatrix},
            \qquad \mbox{for $\zeta \in \Gamma_1$,}\label{JumpPsi1}
           \\
        \Psi_{\bt,+}(\zeta)
        &=&
            \Psi_{\bt,-}(\zeta)
            \begin{pmatrix}
                0 & e^{i\pi\bt} \cr
                -e^{-i\pi\bt} & 0
            \end{pmatrix},
            \qquad \mbox{for $\zeta \in \Gamma_4$,}\\
        \Psi_{\bt,+}(\zeta)
        &=& 
            \Psi_{\bt,-}(\zeta)
            \begin{pmatrix}
                1 & 0 \cr
                e^{i\pi\bt}  & 1
            \end{pmatrix},
            \qquad \mbox{for $\zeta \in \Gamma_2\cup\Gamma_6$,} \\
                \Psi_{\bt,+}(\zeta)
        &=& 
            \Psi_{\bt,-}(\zeta)
            \begin{pmatrix}
                1 & 0 \cr
                e^{-i\pi\bt}  & 1
            \end{pmatrix},
            \qquad \mbox{for $\zeta \in \Gamma_3\cup\Gamma_5$.}
\label{JumpPsi6}
    \end{eqnarray}
\end{enumerate}

We will solve this problem explicitely in terms of the confluent
hypergeometric function, $\psi(a,c;z)$ for the
parameters $a =\beta$ and $c =1$. For the reader's convenience
we briefly sketch the standard theory of the function $\psi(a,c;z)$
in the appendix. 

Introducing the notation
\be\label{psidef}
\psi(a,z) \equiv \psi(a,1;z),
\ee
we define the following piecewise analytic matrix-valued function on the punctured
complex plane $\zeta \in {\mathbb C}\setminus \{0\}$:
\be\label{PsiConfl1}
\Psi(\zeta)=\begin{pmatrix}\psi(\bt,\ze)e^{2\pi i\bt}e^{-\ze/2} &
-\psi(1-\bt,e^{-i\pi}\ze)e^{\pi i\bt}e^{\ze/2}{\Gamma(1-\bt)\over\Gamma(\bt)}
\cr
-\psi(1+\bt,\ze)e^{\pi i\bt}e^{-\ze/2}{\Gamma(1+\bt)\over\Gamma(-\bt)} &
\psi(-\bt,e^{-i\pi}\ze)e^{\ze/2}\end{pmatrix},
\ee
$$
0<\arg \zeta < \pi,
$$
for $\Im \zeta > 0$, and 
\be\label{PsiConfl2}
\eqalign{
\Psi(\zeta)=\begin{pmatrix}\psi(\bt,\ze)e^{2\pi i\bt}e^{-\ze/2} &
-\psi(1-\bt,e^{-i\pi}\ze)e^{\pi i\bt}e^{\ze/2}{\Gamma(1-\bt)\over\Gamma(\bt)}
\cr
-\psi(1+\bt,\ze)e^{\pi i\bt}e^{-\ze/2}{\Gamma(1+\bt)\over\Gamma(-\bt)} &
\psi(-\bt,e^{-i\pi}\ze)e^{\ze/2}\end{pmatrix}\\
\times \begin{pmatrix}1 & 0\cr 2i\sin{\pi \beta} & 1\end{pmatrix},
\qquad \pi < \arg \zeta < 2\pi,}
\ee

for $\Im \zeta < 0$. \footnote{We emphasize that, while the confluent
hypergeometric function $\psi(\beta,\zeta)$ is defined on the
universal covering of the punctured $\zeta$ - plane (see the appendix), 
the function $\Psi(\zeta)$ is defined on the punctured $\zeta$ - plane
itself. Thus the indication of the
ranges of the argument of $\zeta$ in the equations (\ref{PsiConfl1})
and (\ref{PsiConfl2}),
is only relevant for the right-hand sides of the equations.}  
Denote by Roman numerals the sectors among the cuts in Figure 2.
We have the following 

\begin{proposition}\la{Prop2}
A solution to the above RHP (a), (b) for $\Psi_\bt$
is given by the following formulas:
  \begin{eqnarray} 
        \Psi_{\bt}(\zeta)
        &=&
            \Psi(\zeta),
            \qquad \mbox{for $\zeta \in I\cup III$,}\label{solmod1}
           \\
        \Psi_{\bt}(\zeta)
        &=&
            \Psi(\zeta)
            \begin{pmatrix}
                1 & 0\cr
                e^{i\pi\bt\ep} & 1
            \end{pmatrix},
            \qquad \mbox{for $\zeta \in II $,}\label{solmod2}\\
        \Psi_{\bt}(\zeta)
        &=& 
            \Psi(\zeta)
            \begin{pmatrix}
                0 & -e^{i\pi \bt} \cr
                e^{-i\pi\bt}  & 0
            \end{pmatrix},
            \qquad \mbox{for $\zeta \in IV$,} \\
                \Psi_{\bt}(\zeta)
        &=& 
            \Psi(\zeta)
            \begin{pmatrix}
                1 & -e^{-i\pi\bt\ep}\cr
                e^{i\pi\bt\ep}  & 0
            \end{pmatrix},
            \qquad \mbox{for $\zeta \in V $,}\\
              \Psi_{\bt}(\zeta)
        &=& 
            \Psi(\zeta)
            \begin{pmatrix}
                0 & -e^{-i\pi \bt} \cr
                e^{i\pi\bt}  & 0
            \end{pmatrix},
            \qquad \mbox{for $\zeta \in VI$,}
\label{modelsol}
    \end{eqnarray} 
where 
$$
\ep=\begin{cases}1,& \Im \zeta>0\cr -1,& \Im \zeta<0\end{cases}.
$$
\end{proposition}

\begin{proof} 
Equations (\ref{solmod1}) -- (\ref{modelsol}) yield  the correct
jumps across all the contours $\Gamma_{j}$ for any function
$\Psi(\zeta)$ which is continuous in the upper and lower $\zeta$-planes.
Let us show that the choice (\ref{PsiConfl1}) -- (\ref{PsiConfl2})
implies the continuity of $\Psi_{\bt}(\zeta)$ across the real axis.
In the case of the negative semi-axes, this is trivial. Indeed,
for $\zeta < 0$, in both (\ref{PsiConfl1})  and  (\ref{PsiConfl2})
we have to take the same value of $\arg \zeta$, namely $\arg z = \pi$,
and hence we conclude at once that 
\be
\Psi_{+}(\zeta) = \Psi_{-}(\zeta)
\begin{pmatrix}1 & 0\cr -2i\sin{\pi \beta} & 1\end{pmatrix},\quad \mbox{for}\,\,\,\,
\zeta < 0
\ee
(the real axis is oriented from the left to the right).
This relation, in view of (\ref{solmod2}), yields the continuity 
of $\Psi_{\bt}(\zeta)$ across the negative semi-axes.

It is in order to establish the continuity of $\Psi_{\bt}(\zeta)$ for $\zeta > 0$,
that we need the definition (\ref{psidef}) of $\psi(\beta,\zeta)$ as a
confluent hypergeometric function.
The following two fundamental properties of the
function $\psi(\beta,\zeta)$ will play the central role:
\begin{equation}\label{prop1}
\psi(\bt,e^{2\pi i}\zeta) = e^{-2i\pi \bt}\psi(\bt,\zeta) 
+e^{-i\pi \bt}\frac{2\pi i}{\Gamma^{2}(\bt)}
\psi(1-\bt,e^{i\pi}\zeta)e^{\zeta},
\end{equation}
and, by replacing $\beta$ with $1-\beta$ and $e^{i\pi}\ze$ with $\ze$,
\begin{equation}\label{prop3}
\psi(1-\bt,e^{i\pi }\zeta) = e^{2i\pi \bt}\psi(1-\bt,e^{-i\pi}\zeta)
-e^{i\pi \bt}\frac{2\pi i}{\Gamma^{2}(1-\bt)}
\psi(\bt,\zeta)e^{-\zeta}.
\end{equation}
These  equations hold on the universal covering of the punctured
$\zeta$-plane, i.e. for all values of $\arg \zeta$, and they follow from the
general relation (\ref{rep1}) proven,
for the reader's convenience, in the appendix.  

In the case $\zeta > 0$,  the equation $\Psi_{\bt, +}(\zeta) =
\Psi_{\bt, -}(\zeta)$ is equivalent to the equation
\begin{equation}\label{zeta+}
\Psi_{+}(\zeta) = \Psi_{-}(\zeta)
\begin{pmatrix}e^{2i\pi\bt} & -2i\sin(\pi\bt)\cr
0 & e^{-2i\pi\bt}\end{pmatrix}
\end{equation}
(the real axis is oriented from the left to the right). 
For the entry (11) we have,
\be\label{arg2pi}
\eqalign{
\Psi_{11-}(\zeta)e^{2i\pi\bt}
= \psi(\bt,\ze)e^{4\pi i\bt}e^{-\ze/2}\\
-2i\sin(\pi\bt)
\psi(1-\bt,e^{-i\pi}\ze)e^{3\pi i\bt}e^{\ze/2}{\Gamma(1-\bt)\over\Gamma(\bt)},
\qquad \arg\zeta = 2\pi,}
\ee
or
\be\label{arg0}
\eqalign{
\Psi_{11-}(\zeta)e^{2i\pi\bt}
= \psi(\bt,e^{2\pi i}\ze)e^{4\pi i\bt}e^{-\ze/2}\\
-2i\sin(\pi\bt)
\psi(1-\bt,e^{i\pi}\ze)e^{3\pi i\bt}e^{\ze/2}{\Gamma(1-\bt)\over\Gamma(\bt)},
\qquad \arg\zeta = 0.}
\ee
Eliminating, with the help of (\ref{prop1}), the function
$\psi(\bt,e^{2\pi i}\ze)$ from the last equation, and
taking into account the classical relation
$\sin\pi s= \pi/({\Gamma(s)\Gamma(1-s)})$,
we obtain that
\be\eqalign{
\Psi_{11-}(\zeta)e^{2i\pi\bt}
= \psi(\bt,\ze)e^{2\pi i\bt}e^{-\ze/2} 
+ e^{3i\pi \bt}\frac{2\pi i}{\Gamma^{2}(\bt)}
\psi(1-\bt,e^{i\pi}\zeta)e^{\zeta/2}\\
-2i\sin(\pi\bt)
\psi(1-\bt,e^{i\pi}\ze)e^{3\pi i\bt}e^{\ze/2}{\Gamma(1-\bt)\over\Gamma(\bt)}
\equiv \psi(\bt,\ze)e^{2\pi i\bt}e^{-\ze/2},\qquad \arg\zeta = 0,}\label{11}
\ee
which, in view of the definition (\ref{PsiConfl1}), coincides with $\Psi_{11+}(\ze)$
for $\zeta >0$. 

The entry (12) in the r.h.s. of (\ref{zeta+}) reads,
$$
-2i\sin(\pi\bt)\Psi_{11-}(\ze) + \Psi_{12-}(\ze)e^{-2i\pi\bt} 
$$
\be\label{121}
= -2i\sin(\pi\bt) \psi(\bt,\ze)e^{-\ze/2} 
-\psi(1-\bt,e^{i\pi}\ze)e^{-\pi i\bt}e^{\ze/2}{\Gamma(1-\bt)\over\Gamma(\bt)},
\qquad \arg \ze = 0.
\ee
where we used expression (\ref{11})
for $\Psi_{11-}(\ze)$,  and the entry  $\Psi_{12-}(\ze)$ was
taken from (\ref{PsiConfl2}) with the substitution
$\zeta \to e^{2\pi i}\zeta$ (as in the transition from (\ref{arg2pi})
to (\ref{arg0})).

Substituting (\ref{prop3}) for $\psi(1-\bt,e^{i\pi}\ze)$, we obtain that 
(\ref{121}) equals
\be
-\frac{\Gamma(1-\bt)}{\Gamma(\bt)}e^{i\pi\bt}e^{\ze/2}
\psi(1-\bt,e^{-i\pi}\ze),
\qquad \arg \ze = 0,
\ee
which is  $\Psi_{12+}(z)$ for $\zeta >0$.

Consider now the entry (21) of equation (\ref{zeta+}).
As for the entry (11), we first write 
the r.h.s. of the entry (21) as
\be\label{211}
\eqalign{
\Psi_{21-}(\zeta)e^{2i\pi\bt}
= -\psi(1+\bt,e^{2\pi i}\ze)e^{3\pi i\bt}e^{-\ze/2}
\frac{\Gamma(1+\bt)}{\Gamma(-\bt)}\\
+2i\sin(\pi\bt)
\psi(-\bt,e^{i\pi}\ze)e^{2\pi i\bt}e^{\ze/2},
\qquad \arg\zeta = 0.}
\ee
Now use (\ref{prop1}), with $\beta$ replaced by $1+\bt$
to eliminate  $\psi(1+\bt,e^{2\pi i}\ze)$:
$$
\Psi_{21-}(\zeta)e^{2i\pi\bt}
= - e^{3\pi i\bt}e^{-\ze/2}
\frac{\Gamma(1+\bt)}{\Gamma(-\bt)}
\Bigl( \psi(1+\bt,\ze)e^{-2\pi i\bt}
$$
\be
- e^{-i\pi \bt}\frac{2\pi i}{\Gamma^{2}(1+\bt)}
\psi(-\bt,e^{i\pi}\zeta)e^{\zeta}\Bigr)
+2i\sin(\pi\bt)
\psi(-\bt,e^{i\pi}\ze)e^{2\pi i\bt}e^{\ze/2}
\ee
$$
=- e^{3\pi i\bt}e^{-\ze/2}
\frac{\Gamma(1+\bt)}{\Gamma(-\bt)}
\Bigl( \psi(1+\bt,\ze)e^{-2\pi i\bt}
- e^{-i\pi \bt}\frac{2\pi i}{\Gamma^{2}(1+\bt)}
\psi(-\bt,e^{i\pi}\zeta)e^{\zeta}\Bigr)
$$
\be
-\frac{2\pi i}{\Gamma(1+\bt)\Gamma(-\bt)}
\psi(-\bt,e^{i\pi}\ze)e^{2\pi i\bt}e^{\ze/2}
\ee
\be\label{212}
\equiv - e^{\pi i\bt}e^{-\ze/2}
\frac{\Gamma(1+\bt)}{\Gamma(-\bt)}
 \psi(1+\bt,\ze),\qquad \arg\zeta = 0,
\ee
which, in view of the definition (\ref{PsiConfl1}), coincides with $\Psi_{21+}(z)$
for $\zeta >0$. 

Finally, for the entry (22) of (\ref{zeta+}) we have
$$
-2i\sin(\pi\bt)\Psi_{21-}(\ze) + \Psi_{22-}(\ze)e^{-2i\pi\bt} 
$$
\be\label{221}
= 2i\sin(\pi\bt)e^{-\pi i\bt}e^{-\ze/2}
\frac{\Gamma(1+\bt)}{\Gamma(-\bt)}
 \psi(1+\bt,\ze) 
+\psi(-\bt,e^{i\pi}\ze)e^{-2\pi i\bt}e^{\ze/2},
\qquad \arg \ze = 0,
\ee
where   we used   expression (\ref{212})
for $\Psi_{21-}(\ze)$,  and the entry  $\Psi_{22-}(\ze)$ was
taken from (\ref{PsiConfl2}) with the substitution
$\zeta \to e^{2i\pi}\zeta$ (cf. (\ref{121})).
Applying (\ref{prop3}) with $\beta$ replaced by $-\beta$ to
$\psi(-\bt,e^{i\pi}\ze)$, we obtain that (\ref{221}) equals
\be
e^{\ze/2}\psi(-\bt,e^{-i\pi}\ze), \qquad \arg \ze = 0,
\ee
which is  $\Psi_{22+}(z)$ for $\zeta >0$.
This completes the proof of Proposition \ref{Prop2}. 
\end{proof}

We will now show that this solution can be matched 
with $P_\infty(z)$ 
on the boundary $\partial U_{\lb_0}$ for large $n$.
We need first to compute the asymptotic expansion of $\Psi_\bt(\ze)$ for
$\ze\to\infty$. For that we use the classical result \cite{BE} (see also the appendix)
for the confluent hypergeometric function:
\be\label{asclass}
\psi(a,x)=x^{-a}-a^2 x^{-a-1}+O(x^{-a-2}),\qquad |x|\to\infty,\qquad
-3\pi/2<\arg x<3\pi/2.
\ee
We can apply this result  to (\ref{PsiConfl1})
to obtain the asymptotics of the 
solution in sector $I$. The ``correct'' triangular structure of the 
right matrix factor in (\ref{solmod2}) implies
that these asymptotics 
remain the same in the whole second quadrant, namely:
\be\la{Psias}\eqalign{
\Psi_\bt(\ze)=
\left[ I+{1\over\ze}\begin{pmatrix}-\bt^2& -f(\bt)e^{i\pi\bt}\cr 
f(-\bt)e^{-i\pi\bt}& \bt^2\end{pmatrix}+
O(\ze^{-2})\right]\\
\times
\ze^{-\bt\si_3}\begin{pmatrix}e^{2\pi i\bt} & 0\cr 0 & e^{-i\pi\bt}\end{pmatrix}
e^{-\ze\si_3/2},\qquad \ze\to\infty,\qquad 
\frac{\pi}{2}\leq\arg \zeta \leq \pi,}
\ee
where we denoted
\be
f(a)=-{\Gamma(1-a)\over\Gamma(a)}.
\ee
Similar considerations in sector $VI$ and in the upper half of
sector $V$, where we still can obtain the asymptotics
by direct substitution of (\ref{asclass}) into (\ref{PsiConfl1}),
yield
\be\la{Psias2}\eqalign{
\Psi_\bt(\ze)=
\left[ I+{1\over\ze}\begin{pmatrix}-\bt^2& -f(\bt)e^{i\pi\bt}\cr 
f(-\bt)e^{-i\pi\bt}& \bt^2\end{pmatrix}+
O(\ze^{-2})\right]\ze^{-\bt\si_3}\begin{pmatrix}0 & -e^{i\pi\bt}\cr 1 & 0\end{pmatrix}
e^{\ze\si_3/2},\\  \ze\to\infty,\qquad 0\leq \arg \zeta \leq \frac{\pi}{2}.}
\ee

Before discussing the asymptotics of $\Psi_\bt(\ze)$ in the sectors
in the lower half-plane, we need to
perform the deck transformation, $\ze \to e^{-2\pi i}\ze$, with the first column 
of the first matrix factor in
the r.h.s. of (\ref{PsiConfl2}). Using (\ref{prop1}),
$$
\psi(\beta,\zeta)e^{2\pi i\bt}e^{-\zeta/2} = 
\psi(\beta,e^{2\pi i}(e^{-2\pi i}\zeta))e^{2\pi i\bt}e^{-\zeta/2}
$$
$$
= \Bigl(e^{-2i\pi \bt}\psi(\bt,e^{-2\pi i}\zeta) 
+e^{-i\pi \bt}\frac{2\pi i}{\Gamma^{2}(\bt)}
\psi(1-\bt,e^{-i\pi}\zeta)e^{\zeta}\Bigr)e^{2\pi i\bt}e^{-\zeta/2}
$$
$$
= \psi(\bt,e^{-2\pi i}\zeta)e^{-\zeta/2} 
+e^{i\pi \bt}\frac{2\pi i}{\Gamma^{2}(\bt)}
\psi(1-\bt,e^{-i\pi}\zeta)e^{\zeta/2}
$$
\be\label{deck1}
= \psi(\bt,e^{-2\pi i}\zeta)e^{-\zeta/2} 
+2i\sin(\pi\bt)
\psi(1-\bt,e^{-i\pi}\zeta)e^{i\pi \bt}e^{\zeta/2}
\frac{\Gamma(1-\bt)}{\Gamma(\bt)}
\ee
and
$$
-\psi(1+\beta,\zeta)e^{\pi i\bt}e^{-\zeta/2}\frac{\Gamma(1+\bt)}{\Gamma(-\bt)}
 = 
 -\psi(1+\beta,e^{2\pi i}(e^{-2\pi i}\zeta))e^{\pi i\bt}e^{-\zeta/2}\frac{\Gamma(1+\bt)}{\Gamma(-\bt)}
 $$
$$
=- \Bigl(e^{-2i\pi \bt}\psi(1+\bt,e^{-2\pi i}\zeta) 
-e^{-i\pi \bt}\frac{2\pi i}{\Gamma^{2}(1+\bt)}
\psi(-\bt,e^{-i\pi}\zeta)e^{\zeta}\Bigr)e^{\pi i\bt}
e^{-\zeta/2}\frac{\Gamma(1+\bt)}{\Gamma(-\bt)}
$$
\be\label{deck2}
= -\psi(1+\bt,e^{-2\pi i}\zeta)e^{-\pi i\bt}
e^{-\zeta/2}\frac{\Gamma(1+\bt)}{\Gamma(-\bt)} 
-2i\sin(\pi\bt)\psi(-\bt,e^{-i\pi}\zeta)e^{\zeta/2}.
\ee

Equations (\ref{deck1}) and (\ref{deck2}) allow us to rewrite
(\ref{PsiConfl2}) in the form
\be\label{PsiConfl3}
\Psi(\zeta)=\begin{pmatrix}\psi(\bt,e^{-2\pi i}\ze)e^{-\ze/2} &
-\psi(1-\bt,e^{-i\pi}\ze)e^{\pi i\bt}e^{\ze/2}{\Gamma(1-\bt)\over\Gamma(\bt)}
\cr
-\psi(1+\bt,e^{-2\pi i}\ze)e^{-\pi i\bt}e^{-\ze/2}{\Gamma(1+\bt)\over\Gamma(-\bt)} & 
\psi(-\bt,e^{-i\pi}\ze)e^{\ze/2}\end{pmatrix},
\ee
$$
\pi < \arg\ze < 2\pi.
$$
Now the asymptotic result (\ref{asclass}) can be applied directly for the 
whole range $\pi \leq \arg\ze \leq 2\pi$, and as for (\ref{Psias}) and (\ref{Psias2}), we arrive at the
following asymptotics for the function $\Psi_\bt(\ze)$
in the lower half-plane:
\be\la{Psias3}\eqalign{
\Psi_\bt(\ze)=
\left[ I+{1\over\ze}\begin{pmatrix}-\bt^2& -f(\bt)e^{i\pi\bt}\cr 
f(-\bt)e^{-i\pi\bt}& \bt^2\end{pmatrix}+
O(\ze^{-2})\right]\\
\times\ze^{-\bt\si_3}
\begin{pmatrix}e^{2\pi i\bt} & 0\cr 0 & e^{-i\pi\bt}\end{pmatrix}
e^{-\ze\si_3/2},\qquad \ze\to\infty,\qquad 
\pi \leq\arg \zeta \leq \frac{3\pi}{2},}
\ee
and
\be\la{Psias4}\eqalign{
\Psi_\bt(\ze)=
\left[ I+{1\over\ze}\begin{pmatrix}-\bt^2& -f(\bt)e^{i\pi\bt}\cr 
f(-\bt)e^{-i\pi\bt}& \bt^2\end{pmatrix}+
O(\ze^{-2})\right]\\
\times
\ze^{-\bt\si_3}\begin{pmatrix}0 & -e^{3\pi i\bt}\cr e^{-2\pi i\bt} & 0\end{pmatrix}
e^{\ze\si_3/2},\qquad  \ze\to\infty,\qquad \frac{3\pi}{2}\leq \arg \zeta \leq 2\pi.}
\ee

Relations (\ref{Psias}), (\ref{Psias2}),(\ref{Psias3}), and (\ref{Psias4}) give
a complete description of the asymptotic behavior of the function $\Psi_\bt(\ze)$
in the neighborhood of $\ze = \infty$. In particular, for the left and the right
half-planes we have that
\be\la{Psias5}\eqalign{
\Psi_\bt(\ze)=
\left[ I+{1\over\ze}\begin{pmatrix}-\bt^2& -f(\bt)e^{i\pi\bt}\cr 
f(-\bt)e^{-i\pi\bt}& \bt^2\end{pmatrix}+
O(\ze^{-2})\right]\ze^{-\bt\si_3}\begin{pmatrix}e^{2\pi i\bt} & 0\cr 0 & e^{-i\pi\bt}\end{pmatrix}
e^{-\ze\si_3/2},\\ \ze\to\infty,\qquad 
\frac{\pi}{2} \leq\arg \zeta \leq \frac{3\pi}{2},}
\ee
for $\Re \ze\leq 0$ half-plane (upper half-plane in $z$ variable),
and
\be\la{Psias6}\eqalign{
\Psi_\bt(\ze)=
\left[ I+{1\over\ze}\begin{pmatrix}-\bt^2& -f(\bt)e^{i\pi\bt}\cr 
f(-\bt)e^{-i\pi\bt}& \bt^2\end{pmatrix}+
O(\ze^{-2})\right]\ze^{-\bt\si_3}\begin{pmatrix}0 & -e^{\pi i\bt}\cr 1 & 0\end{pmatrix}
e^{\ze\si_3/2},\\  \ze\to\infty,\qquad - \frac{\pi}{2}\leq \arg \zeta \leq \frac{\pi}{2}.}
\ee
for $\Re \ze\geq 0$ half-plane (lower half-plane in $z$ variable).
Substituting these asymptotics into the matching condition (\ref{match}):
\be
P_{\lb_0}(z)P^{-1}_\infty(z)=
E(z)\Psi(\ze)e^{-n\phi(z)\si_3}P^{-1}_\infty(z)=
I+o(1),
\ee
we obtain 
\be\la{E1}
E(z)=P_\infty(z)\ze^{\bt\si_3}e^{n\phi_+(\lb_0)\si_3}
\begin{pmatrix}e^{-2\pi i\bt} &0\cr 0 & e^{i\pi\bt}\end{pmatrix},
\qquad 
\frac{\pi}{2}\leq \arg \zeta \leq \frac{3\pi}{2},
\ee
for $\Im z\geq 0$, and 
\be\la{E2}
E(z)=P_\infty(z)\ze^{-\bt\si_3}e^{-n\phi_+(\lb_0)\si_3}
\begin{pmatrix}0 &1\cr -e^{-i\pi\bt}& 0\end{pmatrix},
\qquad
- \frac{\pi}{2}\leq \arg \zeta \leq \frac{\pi}{2},
\ee
for $\Im z\leq 0$.
A direct check using the jump condition for $P_\infty(z)$ 
shows that $E(z)$ has no jump across the real axis. Moreover, 
$E(z)$ has no singularity at $z=\lb_0$ as easily follows from the expansion
\be\label{Doverze}
{\mathcal{D}(z)\over \ze^\bt}=(8n)^{-\bt}e^{-i\pi\bt}(1-\lb_0^2)^{-3\bt/2}
(1+o(1)),\qquad z\rightarrow \lb_0,\qquad \Im z>0,
\ee
which is derived from (\ref{Dlb}).
Therefore, $E(z)$
is an analytic function in $U_{\lb_0}$.
This completes construction of the parametrix at $\lb_0$:
it is given by the formulas 
(\ref{Plb0},\ref{P1},\ref{E1},\ref{E2}) and  the formulas 
(\ref{solmod1})--(\ref{modelsol}) of Proposition \ref{Prop2}.
In particular, in the pre-image $z(I)$ of the sector $I$ of the $\ze$-plane,
we have:
\be\eqalign{
P_{\lb_0}(z)=P_\infty(z)\ze^{\bt\si_3}\times\\
\begin{pmatrix}\psi(\bt,\ze) &
\psi(1-\bt,e^{-i\pi}\ze)f(\bt)e^{-i\pi\bt}e^{2n\phi_+(\lb_0)}\cr
\psi(1+\bt,\ze)f(-\bt)e^{2i\pi\bt}e^{-2n\phi_+(\lb_0)}&
\psi(-\bt,e^{-i\pi}\ze)e^{i\pi\bt}\end{pmatrix},\quad
z\in z(I),}
\ee
where
\be
f(a)=-{\Gamma(1-a)\over\Gamma(a)},\qquad
\phi_+(\lb_0)=2i\int_{\lb_0}^1\sqrt{1-x^2}dx,\qquad
\ze=4in\int_{\lb_0}^z\sqrt{1-x^2}dx.
\ee
Note that the parametrix $P_{\lb_0}(z)$ correctly reproduces the type 
of singularity of the solution of the RHP of Section 2 at $\lb_0$. Indeed,
on the one hand, it is clear from (\ref{RHM}) that the solution has a 
logarithmic singularity at $z=\lb_0$ (or $\ze=0$). 
On the other hand, the expansion
of the function $\psi(a,1,\ze)$ at zero is known to be \cite{BE}
\be\la{psi0}
\psi(a,1,\ze)= -{1 \over \Gamma(a)}
\left(\ln \ze +{\Gamma'(a)\over \Gamma(a)}-2C_\Gamma\right)+O(\ze\ln\ze),
\qquad \ze\to 0,
\ee
where $C_\Gamma=0.5772\dots$ is Euler's constant.
In fact, $P_{\lb_0}(z)$ reproduces the singularity of the solution at 
$\lb_0$ exactly, in the sense that $S(z)P_{\lb_0}(z)^{-1}$ is analytic
in a neighborhood of $\lb_0$. The latter fact is a consequence of the absence of
jumps and the fact that the singularity can be at most logarithmic.

We can extend
(\ref{match}) into a full asymptotic series in $n$.
For our calculations we need to know the first correction term:
\be\la{alb}\eqalign{
P_{\lb_0}(z)P_\infty^{-1}(z)=I+\De_1(z)+O(1/n^{2-2|\Re\beta|}),\\
\De_1(z)={1\over\ze}P_\infty(z)\begin{pmatrix}-\bt^2& 
-\ze^{2\bt}e^{2n\phi_+(\lb_j)-2\pi i\bt}f(\bt)\cr
\ze^{-2\bt}e^{-2n\phi_+(\lb_j)+2\pi i\bt}f(-\bt)& \bt^2\end{pmatrix}
P_\infty^{-1}(z),\\
\qquad z\in\partial z(I),}
\ee
where $\partial z(I)$ is the 
part of $\partial U_{\lb_0}$ whose $\ze$-image is in $I$.
As the calculation for the other sectors shows,
this expression for $\De_1(z)$ extends by analytic continuation 
to the whole boundary
$\partial U_{\lb_0}$ (cf. \cite{KV,V,Kr2}). 
Moreover, it gives rise to a meromorphic function in 
a neighborhood of $U_{\lb_0}$ with a simple pole at $z=\lb_0$.
The singularity given by $\ze^\bt$ actually cancels with
that of $\mathcal{D}(z)$ from $P_\infty$ (see (\ref{Doverze})).
The error term $O(1/n^{2-2|\Re\beta|})$ in (\ref{alb}) is uniform in $z$ 
on $\partial U_{\lb_0}$
and in $\bt$ in bounded sets of the strip $-1/2<\Re\bt<1/2$.


\subsection{Parametrices at $z=\pm1$}

We now construct parametrices in the remaining regions $U_{1}$ and $U_{-1}$.
These are obtained by a slight generalization
of the results of \cite{Dstrong}, which can be viewed as the case $\om(z)=1$.
The construction is identical with the corresponding one in \cite{Kr2}.

We are looking for an analytic matrix-valued function in $U_1$ which has
the same jump
relation as $S(z)$ there and satisfies the matching 
condition on the boundary:
\be\la{match2}
P_1(z)P_\infty^{-1}(z)=I+O(1/n).
\ee

The solution is:
\be\eqalign{
P_1=E(z)Q(\xi)e^{-n\phi(z)\si_3}\om(z)^{-\si_3/2},\\
E(z)=P_\infty(z)\om(z)^{\si_3/2}e^{i\pi\si_3/4}\sqrt{\pi}
\begin{pmatrix}1&-1\cr 1&1\end{pmatrix}\xi^{\si_3/4}e^{-\pi i/12},}
\ee
and $Q(\xi)$ is given by the expression (7.9) of \cite{Dstrong}
in terms of Airy functions
(in the notation of \cite{Dstrong}, $Q(\xi)=\Psi^\si(\xi)$).
In these formulas
\be
\xi(z)=\left( {3\over 2}n e^{-i\pi}\phi(z)\right)^{2/3}
\ee
is an analytic function 
in a neighborhood of $z=1$ with a cut in the
interval $(-1,1)$, and
\be
\xi(z)=2 n^{2/3} (z-1) \left(1+{1\over 10} (z-1)+ O((z-1)^2)\right).
\ee

The argument of the Airy function on $\partial U_1$ is uniformly large, 
so we can expand it into an 
asymptotic series and proceed the same way as for $\partial U_{\lb_0}$. As a result
we have the matching condition (\ref{match2}) extended to a full asymptotic 
expansion in the inverse powers of $n$.
We shall need only the first 2 terms:
\be\la{a1}\eqalign{
P_1(z)P_\infty^{-1}(z)=I+\De_1(z)+O(1/n^2),\\
\De_1(z)=P_\infty(z)\om(z)^{\si_3/2}e^{\pi i\si_3/4}
{1\over 12}\begin{pmatrix}1/6 & 1\cr -1& -1/6\end{pmatrix}
e^{-\pi i\si_3/4}\om(z)^{-\si_3/2}P_\infty^{-1}(z)
{3\over 2}\xi^{-3/2},\\
z\in \partial U_1.}
\ee

The function $\De_1(z)$ is meromorphic in a neighborhood of $U_1$ with 
a second order pole at $z=1$.


The argument for the parametrix in $U_{-1}$ is similar. 
We have:
\be\la{P-1}\eqalign{
P_{-1}=E(z)\si_3 Q(e^{-i\pi}\xi)\si_3 
e^{-n\tilde\phi(z)\si_3}\om(z)^{-\si_3/2},\\
E(z)=P_\infty(z)\om(z)^{\si_3/2}e^{i\pi\si_3/4}\sqrt{\pi}
\begin{pmatrix}1&1\cr -1&1\end{pmatrix}(e^{-i\pi}\xi)^{\si_3/4}e^{-\pi i/12},}
\ee
Here
\be\eqalign{
\xi(z)=e^{-i\pi}\left( {3\over 2}n \tilde\phi(z)\right)^{2/3},\\
\tilde\phi(z)=
\begin{cases}h(z)/2-i\pi=e^{3i\pi/2}\pi\int_{-1}^z\psi(y)dy,& 
\mbox{for}\quad \Im z>0,\cr
e^{i\pi}h(z)/2+i\pi=
e^{5i\pi/2}\pi\int_{-1}^z\psi(y)dy,& 
\mbox{for}\quad \Im z<0
\end{cases}.}
\ee
For $z\to -1$,
\be
\xi(z)=2 n^{2/3} (1+z) \left(1-{1\over 10} (1+z)+ O((1+z)^2)\right).
\ee

The first 2 terms in the matching condition:
\be\la{am1}\eqalign{
P_{-1}(z)P_\infty^{-1}(z)=I+\De_1(z)+O(1/n^2),\\
\De_1(z)=P_\infty(z)\om(z)^{\si_3/2}e^{\pi i\si_3/4}
{1\over 12}\begin{pmatrix}1/6 & -1\cr 1& -1/6\end{pmatrix}
e^{-\pi i\si_3/4}\om(z)^{-\si_3/2}P_\infty^{-1}(z)\\
\times{3\over 2}(e^{-i\pi}\xi)^{-3/2},\qquad
z\in\partial U_{-1}.}
\ee

As in (\ref{alb}), the error terms in (\ref{a1}) (resp., (\ref{am1})) are uniform for 
all $z\in \partial U_1$ (resp., $z\in \partial U_{-1}$) and $\bt$ in a bounded set.


\subsection{Solving the RHP}

Let
\be
R(z)=\begin{cases}S(z)P_\infty^{-1}(z),& 
z\in U_\infty\setminus\Si,\cr
S(z)P_{\lb_0}^{-1}(z),& 
z\in U_{\lb_0}\setminus\Si,\cr
S(z)P_1^{-1}(z),&
z\in U_1\setminus\Si,\cr
S(z)P_{-1}^{-1}(z),&
z\in U_{-1}\setminus\Si.\end{cases}\la{Rb}
\ee
It is easy to verify that this function has jumps only on 
$\partial U_{\pm1}$, $\partial U_{\lb_0}$, and parts of 
$\Si_j$, $\bbr\setminus [-1,1]$
lying outside the neighborhoods $U_{\pm 1}$, $U_{\lb_0}$ (we
denote these parts without the end-points $\Si^\mathrm{out}$). 
The contour is shown in Figure 3. Outside of it, as a standard argument
shows, $R(z)$ is analytic.
Note that $R(z)=I+O(1/z)$ as $z\to\infty$. 

\begin{figure}
\centerline{\psfig{file=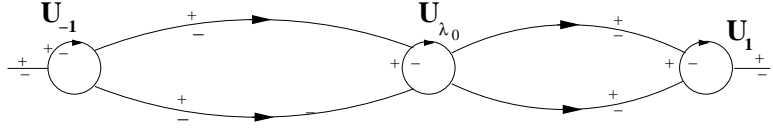,width=4.0in,angle=0}}
\vspace{0cm}
\caption{ 
Contour for the $R$-Riemann-Hilbert problem.}
\label{fig3}
\end{figure}

The jumps are as follows:
\be\eqalign{
R_+(x)=R_-(x)P_{\infty}(x)\begin{pmatrix}1&\om(x)e^{n(g_+(x)+g_-(x)-2x^2-l)}
\cr 0&1\end{pmatrix}
P_{\infty}(x)^{-1},\\ x\in\bbr\setminus [-1-\delta,1+\delta],\\
R_+(x)=R_-(x)P_{\infty}(x)\begin{pmatrix}1&0\cr \om(x)^{-1}
e^{\mp n h(x)}&1\end{pmatrix}
P_{\infty}(x)^{-1},\\ x\in\Si_k^\mathrm{out},
\qquad k=1,\dots,4,\\
\mbox{where the plus sign in the exponent is on 
$\Sigma_{2j}^\mathrm{out}$, and minus, on $\Sigma_{2j-1}^\mathrm{out}$, $j=1,2$,}\\
R_+(x)=R_-(x)P_{\lb_0}(x)P_\infty (x)^{-1},\qquad x\in \partial U_{\lb_0}\setminus
\mbox{\{ intersection points\}},\\
R_+(x)=R_-(x)P_{\pm 1}(x)P_\infty (x)^{-1},\qquad x\in \partial U_{\pm 1}\setminus
\mbox{\{ intersection points\}}.}\la{Rj}
\ee
Here $\delta$ is the radius of $U_1$ and $U_{-1}$.
The jump matrix on $\Si^\mathrm{out}$ can be 
estimated uniformly in $\bt$ as $I+O(\exp(-\ep n|x|))$, where 
$\ep$ is a positive constant.
The jump matrices on $\partial U_{\lb_0,\pm1}$ admit a uniform expansion in the
inverse powers of $n$ multiplied by $n^{2|\Re\bt|}$ (the first 2 terms of which are 
given by (\ref{alb}), (\ref{a1}), and (\ref{am1})):
\be\la{Deas}
I+\De_1(z)+\De_2(z)+\dots+\De_{k}(z)+O(n^{-k-1+2|\Re\beta|}).
\ee
Every $\De_j$ is of order $n^{-j+2|\Re\beta|}$.
(The above expressions give us the explicit form of $\De_1(z)$ in each of 
the neighborhoods.)

The solution to the above RHP is given by a standard analysis (see, e.g.,
\cite{Dstrong, DIKZ, DIK2, Kr2}).
We obtain that
\be\la{Ras}
R(z)=I+\sum_{j=1}^{k}R_j(z)+O(n^{-k-1+2|\Re\beta|}),
\qquad R_j(z)=O(n^{-j+2|\Re\beta|}),\qquad j=1,2\dots
\ee
uniformly for all $z$ and for $\bt$ in  bounded sets 
of the strip $-1/2<\Re\bt<1/2$. 
The expressions for $R_j$ are computed recursively. 
We shall need only the first one:
\be
R_1(z)={1\over 2\pi i}\int_{\partial U}
{\De_1(x)dx\over x-z},\qquad \partial U=\partial U_1\cup\partial U_{-1}\cup\partial U_{\lb_0}.
\la{plem}
\ee
The contours are traversed in the negative direction.


\subsection{Asymptotics for the polynomials and recurrence 
coefficients.}

The asymptotic solution given above allows us to obtain various explicit
formulas for asymptotics of the polynomials orthogonal w.r.t. the weight (\ref{w}).
First of all let us note that such polynomials exist for large enough $n$, i.e.,
the normalizing coefficients $h_k$ are nonzero. Indeed, tracing the transformations
$Y\rightarrow U\rightarrow T\rightarrow S$ of the RHP, we obtain from (\ref{RHM})
\be\la{hnm1}
\eqalign{
h_{n-1}^{-1}=\lim_{z\to\infty}{iY_{21}(\sqrt{2n}z)\over 2\pi (\sqrt{2n}z)^{n-1}}=
\lim_{z\to\infty}{i\sqrt{2n}\over 2\pi z^{n-1} (2n)^n}T_{21}(z)e^{n(g(z)-l)}\\
=\lim_{z\to\infty}{i\sqrt{2n}\over 2\pi (2n)^n}z S_{21}(z)e^{-nl}.}
\ee
Note that
$S(z)=(I+O(1/n^{1-2|\Re\beta|}))P_\infty(z)$; and a simple calculation shows that
\be
\lim_{z\to\infty}zP_{\infty,21}(z)={-i\over 2\mathcal{D}_\infty^2}.
\ee
Therefore, substitution of the expressions for $l$ and $\mathcal{D}_\infty$ into 
(\ref{hnm1}) gives
\be\la{has}
h_n^{-1}=
{e^n 2^n \sqrt{2n}\over 4\pi n^n}e^{-2i\bt\arcsin\lb_0}(1+O(1/n^{1-2|\Re\beta|})).
\ee

We shall now calculate the recurrence coefficients in (\ref{rr}) and prove 
Theorem \ref{Th2}.
Multiplying the recurrence by $\hat p_n(x)w(x)$ and integrating, we obtain 
by orthogonality
\be
A_n h_n=\int_{-\infty}^\infty \hat p_n^2(x)x w(x)dx=
\int_{-\infty}^{\mu_0} \hat p_n^2(x)x e^{-x^2}e^{i\pi\bt}dx+
\int_{\mu_0}^\infty  \hat p_n^2(x)x e^{-x^2}e^{-i\pi\bt}dx.
\ee
Integrating each of these integrals by parts, combining the resulting integrals
together,  and using the orthogonality of $p_n(x)$ and $p'_n(x)$, we obtain
\be\la{A}
A_n=-h_n^{-1} \hat p_n^2(\mu_0) e^{-\mu_0^2} \sinh(i\pi\bt).
\ee 

In terms of the RHP,
\be\la{pn}
\hat p_n(\lb_0\sqrt{2n})=Y_{11}(\lb_0\sqrt{2n})=
(2n)^{n/2} U_{11}(\lb_0).
\ee
In the region $z(I)$ of the $z$-plane, 
\be
U_{11}(\lb_0)=S_{11}(\lb_0)e^{ng_+(\lb_0)}+
S_{12}(\lb_0)e^{ng_-(\lb_0)}e^{i\pi\bt}.
\ee
The main term is obtained here by substituting $P_{\lb_0}(\lb_0)=
\lim_{z\to\lb_0}P_{\lb_0}(z)$ for $S$. Here $z$ tends to $\lb_0$ along a path
in $z(I)$. Expanding $P_{\lb_0}(z)$ we obtain that the main term of $U_{11}(\lb_0)$
equals
\be\la{Up}\eqalign{
\lim_{z\to\lb_0}{1\over 2}
\left(
(a+a^{-1})\mathcal{D}_\infty {\ze^\bt\over \mathcal{D}(z)}
\left[\psi(\bt,\ze)+\psi(1-\bt,e^{-i\pi}\ze)f(\bt)\right]e^{ng_+(\lb_0)}\right.
\\
\left.
-i(a-a^{-1})\mathcal{D}_\infty { \mathcal{D}(z)\over\ze^\bt }
\left[\psi(-\bt,e^{-i\pi}\ze)+\psi(1+\bt,\ze)f(-\bt)\right]
e^{ng_-(\lb_0)+2\pi i\bt}\right).
}
\ee

Recall that $z\to\lb_0$ corresponds to $\ze\to 0$.
In this limit 
the ratio $\mathcal{D}(z)/\ze^\bt$ is given by (\ref{Doverze}), and the behavior of
$\psi$ is determined by (\ref{psi0}). From the latter formula, as $z\to\lb_0$,
\be
\psi(\bt,\ze)+\psi(1-\bt,e^{-i\pi}\ze)f(\bt)=
e^{-i\pi\bt}\Gamma(1-\bt)
+O((z-\lb_0)\ln(z-\lb_0)),
\ee
\be
\psi(-\bt,e^{-i\pi}\ze)+\psi(1+\bt,\ze)f(-\bt)=
e^{-i\pi\bt}\Gamma(1+\bt)
+O((z-\lb_0)\ln(z-\lb_0)).
\ee
A simple computation gives
\be
g_+(\lb_0)=\lb_0^2-{1\over2}-\ln 2 +\phi_+(\lb_0), \qquad
\phi_+(\lb_0)=i\pi\int_{\lb_0}^1\psi(s)ds.
\ee
Substituting these equations into (\ref{Up}) yields
\be\la{Y11}\eqalign{
Y_{11}(\lb_0\sqrt{2n})=(2n)^{n/2} U_{11}(\lb_0)=-(2n)^{n/2}
\mathcal{D}_\infty
2^{-n-1} e^{n(\lb_0^2-1/2)}\bt\\
\times
[(a_+(\lb_0)+a_+(\lb_0)^{-1})F^{(+)}+
i(a_+(\lb_0)-a_+(\lb_0)^{-1})F^{(-)}](1+O(1/n^{1-2|\Re\beta|})),\\
F^{(\pm)}=
(8n)^{\pm\bt} 
(1-\lb_0^2)^{\pm3\bt/2}\Gamma(\mp\bt)
e^{\pm n\phi_+(\lb_0)},\qquad \mathcal{D}_\infty=e^{i\bt\arcsin\lb_0}.}
\ee
By (\ref{A}), (\ref{pn}), this proves the formulas (\ref{An}) for $A_n$ in 
Theorem \ref{Th1}. 
The particular cases (\ref{An2}), (\ref{An3}) of the theorem follow from 
the simple identities:
\[
|\Gamma(\pm i\ga)|^2={\pi\over \ga\sinh\pi\ga},
\qquad \arg\Gamma(-i\ga)=-\arg\Gamma(i\ga).
\qquad \ga\in\bbr\setminus\{0\}.
\]

In order to derive the formulas for $B_n$ in Theorem \ref{Th1} and to
prove Theorem \ref{Th2} using Proposition \ref{Prop1} in the next section, 
we shall need more precise
asymptotics for the coefficients $\ka_n$,
and also those for the  coefficients $\bt_n$, and $\ga_n$ 
of the polynomial $p_n(z)$.
As usual,
we compute them investigating the limit $z\to\infty$ of $Y(z)$. 
By (\ref{RHM}),
\be\la{YU1}
\ka^2_{n-1}=\lim_{z\to\infty}{iY_{21}(z)\over 2\pi z^{n-1}},\qquad
U_{11}(z)=z^n+{\beta_n\over\sqrt{2n}}z^{n-1}+
{\ga_n\over 2n}z^{n-2}+\cdots
\ee
As $z\to\infty$, we need to know asymptotics of $Y(z)$ in the 
area outside the lenses (denote it $A$), which are given by the 
expressions:
\be\eqalign{
Y(z\sqrt{2n})=(2n)^{n\si_3/2}U(z),\\
U(z)=e^{nl\si_3/2}(I+R_1(z)+O(1/n^{2-2|\Re\beta|}))P_\infty(z)
e^{n(g(z)-l/2)\si_3},\\
z\in A\cap U_\infty.}\la{YU2}
\ee

Set
\be\la{LMN}\eqalign{
L_n^{(\pm)}=(8n)^{\pm 2\bt}(1-\lb_0^2)^{\pm 3\bt}f(\pm\bt)e^{\pm 2n\phi_+(\lb_0)},
\qquad
L_n=L_n^{(-)}-L_n^{(+)},\\
M_n=\left(\lb_0+i\sqrt{1-\lb_0^2}\right)L^{(+)}-
\left(\lb_0-i\sqrt{1-\lb_0^2}\right)L^{(-)},\\
N_n=\left(\lb_0-i\sqrt{1-\lb_0^2}\right)L^{(+)}-
\left(\lb_0+i\sqrt{1-\lb_0^2}\right)L^{(-)}.}
\ee

Let us compute $R_1(z)$  using (\ref{plem}). Consider first the 
neighborhood $U_{\lb_0}$.
Substituting $\De_1(x)$ given by (\ref{alb}) into (\ref{plem}) and calculating
residues at a simple pole $x=\lb_0$, we obtain the contribution to $R_1$
from the neighborhood $U_{\lb_0}$:
\be\eqalign{
R_1^{(\lb_0)}(z)=
{1\over 2\pi i}\int_{\partial U_{\lb_0}}{\De_1 dx\over x-z}=
-{1\over z}\left(1+{\lb_0\over z}+O(z^{-2})\right)
{1\over 2\pi i}\int_{\partial U_{\lb_0}}\De_1 dx\\
={1\over z}\left(1+{\lb_0\over z}+O(z^{-2})\right)
{\mathcal{D}_\infty^{\si_3}\over 8n(1-\lb_0^2)}
\begin{pmatrix}2\lb_0\bt^2-iL_n & M_n-2i\bt^2\cr
          N_n-2i\bt^2 & -2\lb_0\bt^2+iL_n\end{pmatrix}
\mathcal{D}_\infty^{-\si_3},}\la{R1lb}
\ee
as $z\to\infty$,
where $L_n$, $M_n$, and $N_n$ are defined in (\ref{LMN}).

To compute the contribution from the neighborhood $U_{1}$, 
we repeat the calculation now using $\De_1(z)$ from (\ref{a1}).
An additional
complication is that the pole at $x=1$ is of
second order. We obtain, using (\ref{calDp1}),
\be\eqalign{
R_1^{(1)}= {1\over 2\pi i}\int_{\partial U_1}{\De_1 dx\over x-z}=
{1\over z}\left(1+{1\over z}+O(z^{-2})\right)\\
\times
{\mathcal{D}_\infty^{\si_3}\over 64n}
\begin{pmatrix}1+8\bt^2{1+\lb_0\over 1-\lb_0}& 4i/3-8i\bt^2{1+\lb_0\over 1-\lb_0}+
8\bt\sqrt{1+\lb_0\over 1-\lb_0}\cr
4i/3-8i\bt^2{1+\lb_0\over 1-\lb_0}-
8\bt\sqrt{1+\lb_0\over 1-\lb_0} & -1-8\bt^2{1+\lb_0\over 1-\lb_0}
\end{pmatrix}\mathcal{D}_\infty^{-\si_3}\\
+{1\over z^2}{5\mathcal{D}_\infty^{\si_3}\over 3\times 64n}
\begin{pmatrix}-1&i\cr i&1\end{pmatrix}\mathcal{D}_\infty^{-\si_3}.}\la{R11}
\ee
Here the dependence on $\bt$ comes from the expansion (\ref{calDp1}).

Using (\ref{calDm1}), a similar calculation for $U_{-1}$ gives 
\be\eqalign{
R_1^{(-1)}= {1\over 2\pi i}\int_{\partial U_{-1}}{\De_1 dx\over x-z}=
{1\over z}\left(1-{1\over z}+O(z^{-2})\right)\\
\times
{\mathcal{D}_\infty^{\si_3}\over 64n}
\begin{pmatrix}-1-8\bt^2{1+\lb_0\over 1-\lb_0}& 4i/3-8i\bt^2{1+\lb_0\over 1-\lb_0}-
8\bt\sqrt{1+\lb_0\over 1-\lb_0}\cr
4i/3-8i\bt^2{1+\lb_0\over 1-\lb_0}+
8\bt\sqrt{1+\lb_0\over 1-\lb_0} & 1+8\bt^2{1+\lb_0\over 1-\lb_0}
\end{pmatrix}\mathcal{D}_\infty^{-\si_3}\\
+{1\over z^2}{5\mathcal{D}_\infty^{\si_3}\over 3\times 64n}
\begin{pmatrix}-1&-i\cr -i&1\end{pmatrix}\mathcal{D}_\infty^{-\si_3}.}
\la{R1-1}
\ee

Summing up the contributions (\ref{R1lb}), (\ref{R11}), 
and (\ref{R1-1}), we obtain
\be
R_1=R_1^{(1)}+R_1^{(-1)}+ R_1^{(\lb_0)}.
\ee
Substituting this into (\ref{YU2}), and using (\ref{calDinf})
and the expansions for $z\to\infty$,
\be
a(z)=1-{1\over 2z}+{1\over 8z^2}+O(z^{-3}),\qquad
g(z)=\ln z - {1\over 8z^2}+O(z^{-4}),\la{auxexp}
\ee
we finally obtain from (\ref{YU1}) 
\be
\ka_{n-1}^2={2^{n-1}e^{-2i\bt\arcsin\lb_0}\over\sqrt{\pi}(n-1)!}
\left\{
1+{iN_n\over 4n(1-\lb_0^2)}
+{\bt^2\over 2n}{2+\lb_0^2\over 1-\lb_0^2}+O\left({1\over
  n^{2-2|\Re\beta|}}\right)\right\},\la{k-as}
\ee
\be
\beta_n=\sqrt{2n}\left\{
i\bt\sqrt{1-\lb_0^2}+
{3\lb_0\bt^2-iL_n/2 \over  4n(1-\lb_0^2)}+
O\left({1\over  n^{2-2|\Re\beta|}}\right)
\right\},\la{b-as}
\ee
\be
\eqalign{
\ga_n=n\left\{-{n-1\over 4}-
\bt^2(1-\lb_0^2)+i\bt\lb_0\sqrt{1-\lb_0^2}+
{1 \over  4n(1-\lb_0^2)}\left(
3\lb_0^2\bt^2+iN_n/2\right.\right.\\
\left.\left.
+2i\bt\sqrt{1-\lb_0^2}[3\lb_0\bt^2-iL_n/2]\right)
+O\left({1\over  n^{2-2|\Re\beta|}}\right)\right\},}\la{g-as}
\ee
where $L_n$ and $N_n$  
are defined in (\ref{LMN}).
The error terms here and in (\ref{Y11}) are uniform 
for $\bt$ in a bounded set provided only $-1/2<\Re\bt<1/2$.
Indeed, 
these terms in (\ref{k-as}) -- (\ref{g-as}) are those in (\ref{YU2}), at worst 
multiplied by a polynomial in $\bt$ (independent of $n$) 
coming from the expansion (\ref{calDinf}) of $D(z)$. 
Now the error term in (\ref{YU2}) is that from (\ref{Ras})
which has the above uniformity property.
A similar analysis holds for the error term in (\ref{Y11}).

The combination of (\ref{k-as}) and (\ref{B}) gives the asymptotics (\ref{Bn})
of Theorem \ref{Th2} (see the remark following (\ref{idas1}) below), which 
reduce to (\ref{Bn2}), (\ref{Bn3}) in the particular cases
$\bt=i\ga$, and $\bt=i\ga$, $\lb_0=0$. Thus, Theorem \ref{Th2} is proved.

We constructed a solution to the Riemann-Hilbert problem of 
Section 2 for $n>n_0$,
$\bt$ in any bounded set of the strip
$-1/2<\Re\bt<1/2$. By uniqueness, it gives the orthogonal 
polynomials via (\ref{RHM}).
On the other hand, the determinantal representation for the orthogonal polynomials
shows that $R(z)$ is an analytic function of $\bt$. 
Furthermore, $R_1(z)$ is
also an analytic function of $\bt$ by construction. 
Thus, the error term in (\ref{Ras})
is both analytic and uniform in $\bt$ in a bounded set of the strip
$-1/2<\Re\bt<1/2$. Therefore, it is differentiable in $\bt$
 (the derivative being uniform in $\bt$). 
Hence, we easily conclude that the error terms in 
(\ref{Y11}), (\ref{k-as})--(\ref{g-as}) have the same differentiability property.
Alternatively, we could have deduced the differentiability of the error terms by 
noticing first that the asymptotic expansions of hypergeometric functions we used 
are differentiable in $\bt$.


\section{Evaluation of the Hankel determinant}

Let us obtain the asymptotic expression for the r.h.s. of the 
differential identity (\ref{id}).
The expression for $Y_{11}(\lb_0\sqrt{2n})$ was found in the previous section and
is given by (\ref{Y11}). 
A similar representation holds for $Y_{21}(\lb_0\sqrt{2n})$:
\be\la{Y21}\eqalign{
Y_{21}(\lb_0\sqrt{2n})=(2n)^{-n/2}U_{21}(\lb_0)=
-(2n)^{-n/2}{1\over \mathcal{D}_\infty}
2^{n-1} e^{n(\lb_0^2+1/2)}\bt\\
\times
[i(a_+(\lb_0)-a_+(\lb_0)^{-1})F^{(+)}-
(a_+(\lb_0)+a_+(\lb_0)^{-1})F^{(-)}](1+O(1/n^{1-2|\Re\beta|})).}
\ee
Since our asymptotic formulas are differentiable in $\bt$,
we obtain from (\ref{Y11}) and (\ref{Y21}):
\be\eqalign{
Y_{11}Y'_{21,\,\bt} -Y'_{11,\,\bt} Y_{21}=
-2{\mathcal{D}'_{\infty,\,\bt}\over \mathcal{D}_\infty}Y_{11}Y_{21}\\
-\bt e^{2n\lb_0^2}\left\{F^{(+)\,'}_\bt F^{(-)}-
F^{(+)}F^{(-)\,'}_\bt+O\left({\ln n\over n^{1-4|\Re\beta|}}\right)\right\}=
-2{\mathcal{D}'_{\infty,\,\bt}\over \mathcal{D}_\infty}Y_{11}Y_{21}\\
-{\pi\bt\over \sin\pi\bt} e^{2n\lb_0^2}\left\{
2\ln(8n(1-\lb_0^2)^{3/2})-\left(\ln{\Gamma(\bt)\over \Gamma(-\bt)}\right)'_\bt
+O\left({\ln n\over n^{1-4|\Re\beta|}}\right)\right\}.}
\ee
As follows from (\ref{has}),
\[
[\ln(\ka_n\ka_{n-1})]'_\bt=
-2{\mathcal{D}'_{\infty,\,\bt}\over \mathcal{D}_\infty}
+O\left({\ln n\over n^{1-2|\Re\beta|}}\right).
\]
Therefore we have for the last part of (\ref{id}):
\be\la{idas1}\eqalign{
{1\over\pi}e^{-\mu_0^2}\sin{(\pi\bt)}
\left[Y_{11}(\mu_0)Y_{21}(\mu_0)'_\bt-
Y_{11}(\mu_0)'_{\bt} Y_{21}(\mu_0)
-(\ln\ka_n\ka_{n-1})'_{\bt}Y_{11}(\mu_0) Y_{21}(\mu_0)\right]\\
=-\bt\left\{
2\ln(8n(1-\lb_0^2)^{3/2})-\left(\ln{\Gamma(\bt)\over \Gamma(-\bt)}\right)'_\bt
\right\}+O\left({\ln n\over n^{1-4|\Re\beta|}}\right).}
\ee

We now turn to the evaluation of the rest of (\ref{id}).
Care is needed with the estimation of
$\ka_n$. To obtain the 
asymptotics of $\ka_n$ from (\ref{k-as}) we need first to replace 
$n$ with $n+1$ and second, to replace $\lb_0$ with $\lb_0\sqrt{n\over n+1}$.
After a straightforward calculation, we obtain
\be\la{idas2}\eqalign{
-n(\ln\ka_n\ka_{n-1})'_{\bt}-
\left(\ka_{n-1}^2\over \ka_n^2\right)'_{\bt}+
2\left[\ga'_{n,\bt}-\beta_n\beta'_{n,\bt}\right]\\
=
2n\ln\left(i\lb_0+\sqrt{1-\lb_0^2}\right)-2\bt+2i\lb_0 n\sqrt{1-\lb_0^2}
+O\left({\ln n\over n^{1-2|\Re\beta|}}\right).}
\ee

The sum of (\ref{idas1}) and (\ref{idas2}) yields
\be\la{idfin}\eqalign{
{d\over d\bt}\ln D_n(\bt)=
2in\arcsin\lb_0-2\bt+2i\lb_0 n\sqrt{1-\lb_0^2}\\
-\bt \left\{
2\ln(8n(1-\lb_0^2)^{3/2})-\left(\ln{\Gamma(\bt)\over \Gamma(-\bt)}\right)'_\bt
\right\}
+O\left({\ln n\over n^{1-4|\Re\beta|}}\right).}
\ee

Since the error term here is uniform in $\bt$, we can integrate 
this identity.

Let $\wt\Om$ be a bounded subset of the strip $-1/4<\Re\bt<1/4$ 
(cf. Section 2). The identity (\ref{idfin}) was derived for $\bt\notin\Om$ 
(see Sections 2 and 3). Note that the number of points $\bt$ in $\Om$, if any, 
is finite. Indeed, the function
$\ka_k^2=\ka_k^2(\bt)$ is a ratio $D_k/D_{k+1}$ of two 
analytic functions of $\bt$, which are not identically zero because they are 
known to be positive for $\Re\bt=0$.
Let us rewrite the identity (\ref{idfin}) in 
the form $f'(\bt)=0$, where $f(\bt)=D_n(\bt)\exp(-\int_0^\bt r(n,t)dt)$ and
where $r(n,\bt)$ is the r.h.s. of (\ref{id}).
Since expression (\ref{idfin}) for $r(n,\bt)$ holds uniformly and is continuous 
for $\bt\in\wt\Om$ provided $n$ is larger than some $n_0(\wt\Om)$,
and $D_n(\bt)$ and its derivative are continuous, the function
$f(\bt)$ is continuously differentiable for all $n>n_0(\wt\Om)$. Hence,
$f'(\bt)=0$ for {\it all} $\bt\in\wt\Om$ and $n>n_0(\wt\Om)$. 
Taking into account that $f(\bt)=f(0)=D_n(0)\neq 0$,
we conclude that  $D_n(\bt)$ is nonzero, and that the identity (\ref{idfin})
is, in fact, true for all $\bt\in\wt\Om$ if $n$ is sufficiently large
(larger than $n_0(\wt\Om)$).

Now set $\bt=0$, and integrate (\ref{idfin})
from $0$ to some $\bt$ over $\bt$. 
We obtain
\be\la{fin1}\eqalign{
\ln D_n(\bt)-\ln D_n(0)=
2n\bt\ln\left(i\lb_0+\sqrt{1-\lb_0^2}\right)-
\bt^2+2i\bt\lb_0 n\sqrt{1-\lb_0^2}-\\
\bt^2 \ln(8n(1-\lb_0^2)^{3/2})-\int_0^\bt\ln{\Gamma(\bt)\over \Gamma(-\bt)}d\bt
+\bt \ln{\Gamma(\bt)\over \Gamma(-\bt)}
+O\left({\ln n\over n^{1-4|\Re\beta|}}\right).}
\ee

In view of (\ref{C}) Theorem \ref{Th1} is proved.

\section{Acknowledgement}
We thank Y. Chen for attracting our attention to this problem.

\section{Appendix} 

The confluent hypergeometric function $\psi(a,c;z)$
is defined as a unique solution of the confluent
hypergeometric equation
\begin{equation}\label{che}
zw'' +(c-z)w' -aw = 0
\end{equation}
satisfying the asymptotic condition
\begin{equation}\label{cha}
\psi(a,c;z) \sim z^{-a}\sum_{n=0}^{\infty}
(-1)^{n}\frac{(a)_{n}(1+a-c)_{n}}{n!z^{n}},
\end{equation}
$$
z \to \infty, \qquad -\frac{3\pi}{2} <  \arg z < \frac{3\pi}{2}.
$$
Here the standard notation
$$
(a)_{0} = 1, \qquad (a)_{n} = a(a+1)\cdots (a+n-1) = 
\frac{\Gamma(a+n)}{\Gamma(a)},\qquad n \geq 1
$$
is used. The function $\psi(a,c;z)$ admits the following 
integral representation \cite{BE}:
\begin{equation}\label{chi}
\psi(a,c;z) = \frac{1}{\Gamma(a)}
\int_{0}^{\infty e^{-i\alpha}}t^{a-1}(1+t)^{c-a-1}e^{-zt}dt,
\end{equation}
\begin{equation}\label{chi1}
-\pi < \alpha < \pi, \qquad -\frac{\pi}{2} 
+ \alpha < \arg z < \frac{\pi}{2} + \alpha.
\end{equation} 
In (\ref{chi}) the inequality
\begin{equation}\label{ain}
\Re a > 0
\end{equation}
is assumed, and the branches of the functions $t^{a-1}$
and $(1+t)^{c-a-1}$ are defined on the $t$-plane cut along
the semi-axis $(-\infty, 0]$ and fixed by the conditions,
\begin{equation}\label{branches}
-\pi <\arg t < \pi, \qquad -\pi <  \arg (1+t) < \pi.
\end{equation}

 When $\alpha = 0$, the integration in formula (\ref{chi}) is performed
 along the positive semi-axis, $0\leq t <\infty$, and the formula
 gives $\psi(a,c;z)$ for the values of $\arg z$ between $-\frac{\pi}{2}$
 and $\frac{\pi}{2}$. The increase of  $\alpha$ from $0$ to $\pi$
 leads to the clock-wise rotation of the ray of integration and,
 simultaneously, to the counter-clock-wise rotation of the
 domain in the (universal covering of ) punctured $z$-plane. This process produces  
 an analytic continuation of $\psi(a,c;z)$ to the 
 domain
 \begin{equation}\label{d1}
 \frac{\pi}{2} < \arg z < \frac{3\pi}{2}
 \end{equation}
and is shown in Figure \ref{A1}.

\begin{figure}
\centerline{\psfig{file=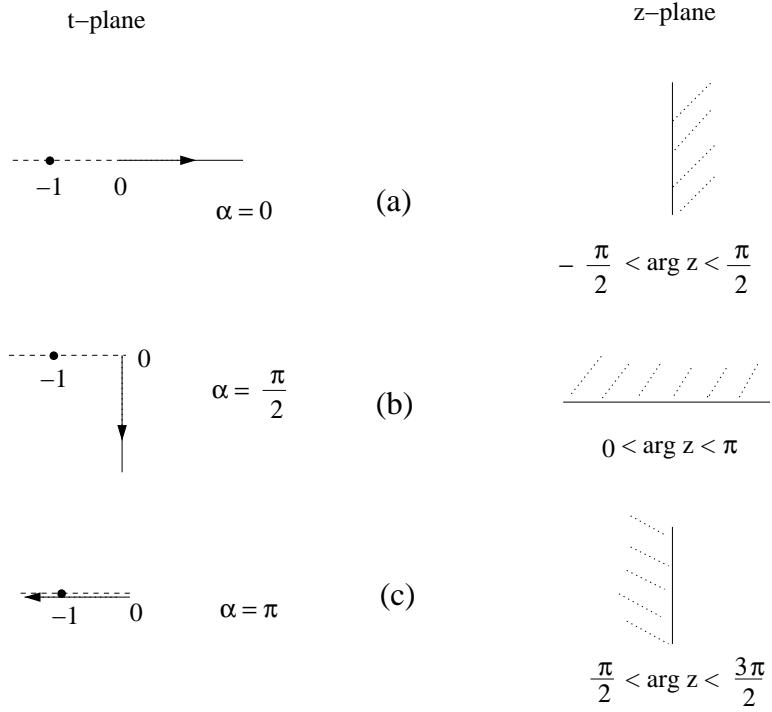,width=4.0in,angle=0}}
\vspace{0cm}
\caption{
Analytic continuation of $\psi(a,c;z)$ to the domain 
$\frac{\pi}{2} < \arg z < \frac{3\pi}{2}$.}
\label{A1}
\end{figure}

When $z$ is in the domain (\ref{d1}), the contour of integration is the lower side
of the cut $(-\infty, 0]$.
Similarly, the decrease of $\alpha$ from $0$ to $-\pi$
 leads to the contour-clock-wise rotation of the ray of integration and,
 simultaneously, to the clock-wise rotation of the
 domain in the (universal covering of ) punctured $z$-plane. This process produces  
 an analytic continuation of $\psi(a,c;z)$ to the 
 domain
\begin{equation}\label{d2}
- \frac{3\pi}{2} < \arg z < -\frac{\pi}{2}
\end{equation}
and is shown in Figure \ref{A2}.
When $z$ is in the domain (\ref{d2}), the contour of integration is the upper side
of the cut $(-\infty, 0]$.

\begin{figure}
\centerline{\psfig{file=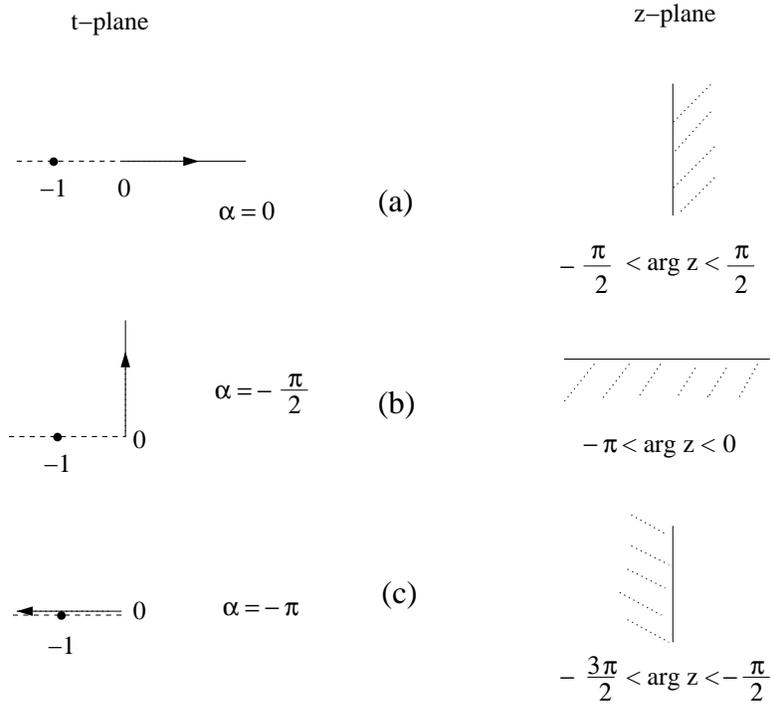,width=4.0in,angle=0}}
\vspace{0cm}
\caption{
Analytic continuation of $\psi(a,c;z)$ to the domain 
$-\frac{3\pi}{2} < \arg z < -\frac{\pi}{2}$.}
\label{A2}
\end{figure}

When $\alpha$ goes beyond the interval $(-\pi, \pi)$, equation (\ref{chi})
provides  the analytical  continuation of the function $\psi(a,c;z)$ to
the whole universal covering, $\widetilde{{\mathbb C}} \setminus \{0\}$, 
of the punctured $z$-plane ${\mathbb C} \setminus \{0\}$. 

As $z$ passes from one sheet of the universal covering
$\widetilde{{\mathbb C}} \setminus \{0\}$ to another, a loop
around the interval $[-1, 0]$ must be added to the contour of integration
in (\ref{chi}). In Figure \ref{A3}, we illustrate the process of the analytic
continuation of $\psi(a,c;z)$ to the sheet marked by the condition
\begin{equation}\label{d3}
\frac{3\pi}{2} < \arg z < \frac{5\pi}{2}.
\end{equation}

\begin{figure}
\centerline{\psfig{file=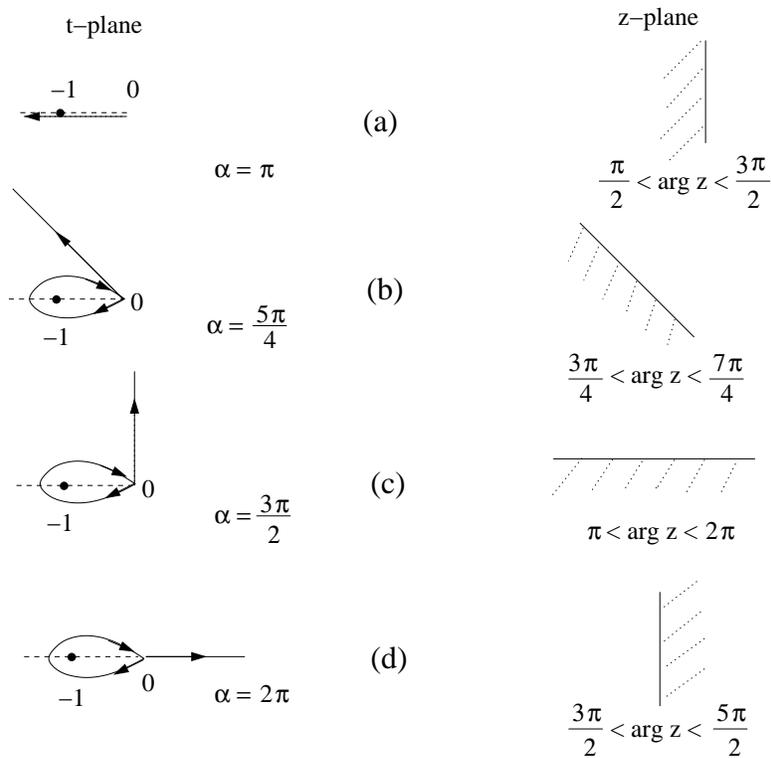,width=4.0in,angle=0}}
\vspace{0cm}
\caption{
Analytic continuation of $\psi(a,c;z)$ to the domain 
$\frac{3\pi}{2} < \arg z < \frac{5\pi}{2}$.}
\label{A3}
\end{figure}

The domain (\ref{d3}) corresponds to the choice 
$\alpha = 2\pi$ in (\ref{chi}) with the contour of
integration as shown in Figure \ref{A3}d. Assume that,
in addition to (\ref{ain}), the inequality
\begin{equation}\label{cin}
\Re c > \Re a
\end{equation}
holds. The loop in Figure \ref{A3}d can be deformed to the
sides of the cut $[-1, 0]$ with the following arguments 
of $t$ and $1+t$ in the integrand in
the r.h.s. of (\ref{chi}):
\be\label{arg}
\eqalign{
\arg t|_{[-1,0]_{-}} = \arg t|_{[-1,0]_{+}} = -\pi, \qquad 
\arg t|_{[0, \infty)} = -2\pi,\\
\arg (1+t)|_{[-1,0]_{-}} = 0, \qquad  
\arg (1+t)|_{[-1,0]_{+}} = 
\arg (1+t)|_{[0, \infty)} = -2\pi.}
\ee
Here we use the notation $[-1,0]_{+}$ and $[-1,0]_{-}$
for the upper and lower sides of the cut $[-1,0]$, respectively.
It follows
from (\ref{arg}) that
the function $\psi(a,c;z)$ on the sheet (\ref{d3}) can be 
represented as
\begin{equation}\label{psi2p1}
\eqalign{
\psi(a,c;z) = -\frac{e^{-i\pi a}}{\Gamma(a)}\int_{0}^{-1}
|t|^{a-1}|1+t|^{c-a-1}e^{-zt}dt\\
- \frac{e^{-i\pi a}}{\Gamma(a)}e^{-2i\pi(c - a)}\int_{-1}^{0}
|t|^{a-1}|1+t|^{c-a-1}e^{-zt}dt\\
+ \frac{e^{-2i\pi a}}{\Gamma(a)}e^{-2i\pi(c - a)}\int_{0}^{\infty}
|t|^{a-1}|1+t|^{c-a-1}e^{-zt}dt, \qquad 
\frac{3\pi}{2} < \arg z < \frac{5\pi}{2},}
\end{equation}
where we assumed the convention that
$\arg|a| = 0$ in any expression of the form $|a|^b$.
Observe now that in the principal sheet, i.e. 
when $-\frac{\pi}{2} < \arg z < \frac{\pi}{2}$,
the function $\psi(a,c;z)$ is given by the original
formulae (\ref{chi}) -- (\ref{branches}) where one has
to set $\alpha = 0$ (see also Figure \ref{A1}a). 
This yields the equation:
\begin{equation}\label{psi2p2}
\psi(a,c;z) = \frac{1}{\Gamma(a)}\int_{0}^{\infty}
|t|^{a-1}|1+t|^{c-a-1}e^{-zt}dt,\qquad
-\frac{\pi}{2} < \arg z < \frac{\pi}{2}.
\end{equation}
Combined, equations (\ref{psi2p1}) and (\ref{psi2p2}) 
imply the following relation between the values of the
functions $\psi(a,c;z)$ and $\psi(a,c;e^{2\pi i}z)$:
\begin{equation}\label{psi2p3}
\psi(a,c;e^{2\pi i}z) = 
e^{-2\pi ic}\psi(a,c;z) +
\frac{e^{-i\pi a}}{\Gamma(a)}\left( 1 - e^{-2\pi i(c - a)}\right)\int_{-1}^{0}
|t|^{a-1}|1+t|^{c-a-1}e^{-zt}dt.
\end{equation}
Note that this equation holds on the {\it whole} universal
covering $\widetilde{{\mathbb C}} \setminus \{0\}$, i.e. for
{\it all} values of $\arg z$.  Indeed, the function
 $\psi(a,c;z)$ has already been defined as a function
 on $\widetilde{{\mathbb C}} \setminus \{0\}$.
 The function  $\psi(a,c;e^{2\pi i}z)$  is a composition of the
 function $\psi(a,c;z)$ and the conformal isomorphism
 of the universal covering $\widetilde{{\mathbb C}} \setminus \{0\}$
 determined by the mapping $z \to e^{2\pi i}z$. Finally,
the integral over the finite interval $[-1, 0]$ is an
entire function and hence an analytic function on the
universal covering  $\widetilde{{\mathbb C}} \setminus \{0\}$ as well. 

The integral term in (\ref{psi2p3}) actually describes the first canonical
solution of equation (\ref{che}) at $z =0$. More precisely, equation
(\ref{che}) possesses a unique normalized holomorphic at zero
solution,
\begin{equation}\label{phidef}
\phi(a,c;z) = \sum_{n=0}^{\infty}\frac{(a)_{n}}{(c)_{n}}\frac{z^n}{n!}.
\end{equation}
The integral representation of $\phi(a,c;z)$ links it with the
integral from (\ref{psi2p3}). Namely,
\begin{equation}\label{phiint}
\phi(a,c;z) = \frac{\Gamma(c)}{\Gamma(a)\Gamma(c-a)}\int_{0}^{1}
t^{a-1}(1-t)^{c-a-1}e^{zt}dt
\end{equation}
$$
\equiv  \frac{\Gamma(c)}{\Gamma(a)\Gamma(c-a)}
\int_{-1}^{0}
|t|^{a-1}|1+t|^{c-a-1}e^{-zt}dt.
$$
Therefore, equation  (\ref{psi2p3}) can be interpreted as the relation
\begin{equation}\label{psiphi1}
\psi(a,c;e^{2\pi i}z) = 
e^{-2\pi ic}\psi(a,c;z) +
e^{-i\pi c}\frac{2\pi i}{\Gamma(c)\Gamma(1+a-c)}
\phi(a,c;z),
\end{equation}
where we used the classical formula
$\Gamma(s)\Gamma(1-s) = \pi/\sin{\pi s}.$

It is worth noticing that the function $\phi(a,c;z)$, as 
an entire function of the variable $z$, is invariant under
the mapping $z \to e^{2\pi i}z$:
\begin{equation}\label{phiaut}
\phi(a,c;e^{\pm 2\pi i}z) = \phi(a,c;z).
\end{equation}

Our next task is to find another
relation between the $\psi$-and $\phi$-functions.
To this end, let
\begin{equation}\label{psihat}
\widehat{\psi}(a,c;z)\equiv e^{-i\pi(c-a)}\psi(c-a,c;e^{-i\pi}z)e^{z}.
\end{equation}
The function $\widehat{\psi}(a,c;z)$ is well-defined by this equation as 
an analytic function on $\widetilde{{\mathbb C}} \setminus \{0\}$.
In fact, $\widehat{\psi}(a,c;z)$ is  the second canonical
at infinity solution of the confluent
hypergeometric equation (\ref{che}). Its behavior
at infinity is given by the asymptotic series
\begin{equation}\label{cha2}
\widehat{\psi}(a,c;z) \sim e^{z}z^{a-c}\sum_{n=0}^{\infty}
\frac{(c-a)_{n}(1-a)_{n}}{n!z^{n}},
\end{equation}
$$
z \to \infty, \qquad -\frac{\pi}{2} <  \arg z < \frac{5\pi}{2}.
$$
Let us establish the relation between the functions $\widehat{\psi}(a,c;z)$, 
$\psi(a,c;z)$, and $\phi(a,c;z)$.

We start with the integral representation for $\widehat{\psi}(a,c;z)$ assuming that
$\frac{3\pi}{2} < \arg z < \frac{5\pi}{2}$. The latter implies that
$$
\frac{\pi}{2} < \arg e^{-i\pi}z < \frac{3\pi}{2},
$$
and hence we can use the definitions
(\ref{chi}) -- (\ref{branches}) with
$\alpha = \pi$ in  (\ref{psihat}).
This leads to the equation:
\begin{equation}\label{psihat2}
\widehat{\psi}(a,c;z)=
\frac{e^{-i\pi(c-a)}}{\Gamma(c-a)}\int_{0}^{-\infty}
t^{c-a-1}(1+t)^{a-1}e^{zt}dt\,e^{z},\qquad
\frac{3\pi}{2} < \arg z < \frac{5\pi}{2},
\end{equation}
where it is assumed that (cf. Figure \ref{A1}c)
\[
\arg t|_{(-\infty, 0]} = -\pi,\qquad
\arg(1+t)|_{(-\infty, -1]} = -\pi, \qquad
\arg(1+t)|_{(-1, 0]} = 0.
\]
Changing the variable, $1+t \to -t$,
and taking
into account the above conditions on the arguments, we
can rewrite (\ref{psihat2}) as
\begin{equation}\label{psihat3}
\widehat{\psi}(a,c;z)=
-\frac{e^{-i\pi(c-a)}}{\Gamma(c-a)}\int_{-1}^{\infty}
(-1-t)^{c-a-1}(-t)^{a-1}e^{-zt}dt,\qquad
\frac{3\pi}{2} < \arg z < \frac{5\pi}{2},
\end{equation}
where 
\be\eqalign{
(-1 - t)^{c-a-1}|_{(-1,\infty]} = |1+t|^{c-a-1}e^{-i\pi(c-a-1)},\\
(-t)^{a-1}|_{(-1, 0]} = |t|^{a-1}, \qquad
(-t)^{a-1}|_{[0, \infty)} = |t|^{a-1}e^{-i\pi(a-1)}.}
\ee
This, in turn, leads to the equation
\begin{equation}\label{psihat4}
\eqalign{
\widehat{\psi}(a,c;z)=
\frac{e^{-2i\pi(c-a)}}{\Gamma(c-a)}\int_{-1}^{0}
|1+t|^{c-a-1}|t|^{a-1}e^{-zt}dt\\
-\frac{e^{-2i\pi(c-a) -i\pi a}}{\Gamma(c-a)}\int_{0}^{\infty}
|1+t|^{c-a-1}|t|^{a-1}e^{-zt}dt, \qquad
\frac{3\pi}{2} < \arg z < \frac{5\pi}{2},}
\end{equation}
which, with the help of (\ref{psi2p2}) and (\ref{phiint}), can be re-written as
\begin{equation}\label{psihat5}
\eqalign{
\widehat{\psi}(a,c;z)=
e^{-2i\pi(c-a)}\frac{\Gamma(a)}{\Gamma(c)}\phi(a,c;z)\\
-e^{-2i\pi c +i\pi a}\frac{\Gamma(a)}{\Gamma(c-a)}\psi(a,c;e^{-2\pi i}z), \qquad
\frac{3\pi}{2} < \arg z < \frac{5\pi}{2}.}
\end{equation}
Note  that the restriction $\frac{3\pi}{2} < \arg z < \frac{5\pi}{2}$ can 
now be dropped; indeed, all the functions involved 
are analytic functions on the whole
universal covering  $\widetilde{{\mathbb C}} \setminus \{0\}$. 
Equation (\ref{psihat5}) is the relation between the functions
$\widehat{\psi}(a,c;z)$, $\psi(a,c;z)$, and $\phi(a,c;z)$ we were looking for.

Equations (\ref{psihat5}) and (\ref{psihat}) yield the formula
\begin{equation}\label{int01}
\phi(a,c;z)
=\frac{\Gamma(c)}{\Gamma(c-a)} e^{-i\pi a}\psi(a,c;e^{-2\pi i}z)
+\frac{\Gamma(c)}{\Gamma(a)} e^{i\pi(c-a) }\psi(c-a,c;e^{-i\pi }z) e^{z},
\end{equation}
which, with the help of the mapping $z \to e^{2\pi i}z$, can be also written as
\begin{equation}\label{int02}
\phi(a,c;z)
=\frac{\Gamma(c)}{\Gamma(c-a)} e^{-i\pi a}\psi(a,c;z)
+\frac{\Gamma(c)}{\Gamma(a)} e^{i\pi(c-a) }\psi(c-a,c;e^{i\pi }z) e^{z}.
\end{equation}
Observe that the operation 
$z\rightarrow e^{-2\pi i}z$ brings  equation   (\ref{psiphi1})  to the form
\begin{equation}\label{psiphi2}
\psi(a,c;e^{-2\pi i}z) = 
e^{2\pi ic}\psi(a,c;z) -
e^{i\pi c}\frac{2\pi i}{\Gamma(c)\Gamma(1+a-c)}
\phi(a,c;z).
\end{equation}
Using this relation in (\ref{int01}), we can exclude $\psi(a,c;e^{-2\pi i}z)$
from the latter and obtain a ``companion'' equation to (\ref{int02}), i.e.,
\begin{equation}\label{int03}
\phi(a,c;z)
=\frac{\Gamma(c)}{\Gamma(c-a)} e^{i\pi a}\psi(a,c;z)
+\frac{\Gamma(c)}{\Gamma(a)} e^{-i\pi(c-a) }\psi(c-a,c;e^{-i\pi }z) e^{z}.
\end{equation}
Note again that all the five relations (\ref{int01}), (\ref{int02}), (\ref{psiphi1}),
(\ref{psiphi2}), and (\ref{int03}) hold on the {\it whole} universal
covering $\widetilde{{\mathbb C}} \setminus \{0\}$, i.e., for
{\it all} values of $\arg z$.

Each of  the equations (\ref{int02}) and (\ref{int03}) allows us to
exclude the function $\phi(a,c;z)$ from (\ref{psiphi1}).
This leads to the following two representations of
the function $\psi(a,c;e^{2\pi i}z)$:
\begin{equation}\label{rep1}
\psi(a,c;e^{2\pi i}z) = e^{-2i\pi a}\psi(a,c;z) 
+e^{-i\pi a}\frac{2\pi i}{\Gamma(a)\Gamma(1+a-c)}
\psi(c-a,c;e^{i\pi}z)e^{z}
\end{equation}
and
\begin{equation}\label{rep2} 
\eqalign{
\psi(a,c;e^{2\pi i}z) = 
\Bigl( 1 + e^{-2i\pi c} - e^{-2i\pi c+2i\pi a}\Bigr)\psi(a,c;z)\\
+e^{i\pi a - 2i\pi c}\frac{2\pi i}{\Gamma(a)\Gamma(1+a-c)}
\psi(c-a,c;e^{-i\pi}z)e^{z}.}
\end{equation}

\begin{remark}
Restrictions (\ref{ain}) and (\ref{cin}) can be lifted. The functions
$\psi(a,c;z)$ and $\phi(a,c;z)/\Gamma(c)$ are in fact entire functions 
of the complex
parameters $a$ and $c$. For the function $\psi(a,c;z)$ this  can be seen in 
the usual way
by replacing, cf. \cite{BE}, the integration in (\ref{chi}) along the
ray by the integration along  the loop around this ray. Namely,
one can rewrite (\ref{chi}) as
\be\label{chigeneral}
\psi(a,c;z) = \frac{1}{\left(1-e^{2\pi i a}\right)\Gamma(a)}
\int_{(0,+)}^{\infty e^{-i\alpha}}t^{a-1}(1+t)^{c-a-1}e^{-zt}dt
\end{equation}
\be\label{chigeneral0}
\equiv
\frac{i}{2\pi}e^{-i\pi a}\Gamma(1-a)
\int_{(0,+)}^{\infty e^{-i\alpha}}t^{a-1}(1+t)^{c-a-1}e^{-zt}dt,
\ee
\begin{equation}\label{chi1general}
-\frac{\pi}{2}+ \alpha < \arg z < \frac{\pi}{2} + \alpha,
\end{equation} 
where the symbol 
$\int_{(0,+)}^{\infty e^{-i\alpha}}$
stands for the integration along the loop around the
original contour in (\ref{chi}). In Figure \ref{A4}, the new contour
is shown for several different values of the parameter $\alpha$.

\begin{figure}
\centerline{\psfig{file=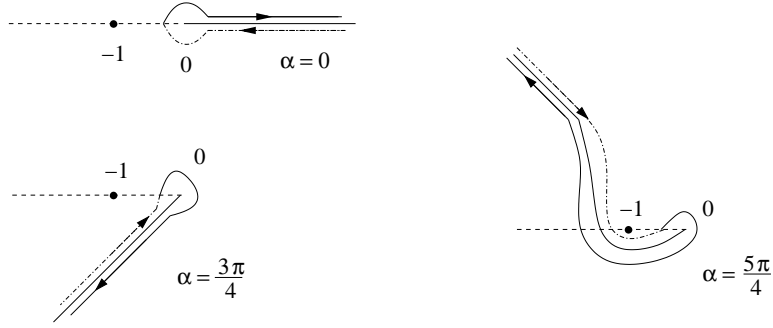,width=4.0in,angle=0}}
\vspace{0cm}
\caption{
Contours of integration for  $\psi(a,c;z)$ with arbitrary
$a$ and $c$.}
\label{A4}
\end{figure}

Equations (\ref{chigeneral}) -- (\ref{chi1general}) define the analytic
continuation of $\psi(a,c;z)$ into the whole $a$- and $c$-complex
planes (note that the integer values of $a$ are removable
singularities of the right-hand sides in (\ref{chigeneral}), (\ref{chigeneral0})).
Simultaneously, the validity of the relations (\ref{rep1}) and (\ref{rep2})
is extended to all complex $a$ and $c$.

Also in the usual way (see e.g. \cite{Olver}, Chapter 4, Section 5), 
the representations (\ref{chigeneral}) -- (\ref{chi1general}) 
can be used to ensure that the asymptotics (\ref{cha}) are uniform
for $a$ and $c$ belonging to any compact set of the complex plane.

In the case of the function $\phi(a,c;z)/\Gamma(c)$,
the relevant analytic continuation is achieved 
with the help of the replacement in (\ref{phiint}) of the integration 
along the interval $[0,1]$ by the integration over the double loop:
Pochhammer's loop (see, e.g., \cite{BE}).  Alternatively, one can just use the
relation  (\ref{int02}) and already established analyticity
of the $\psi$-function.
\end{remark}

\begin{remark}\label{BE}
In the standard literature, see e.g. \cite{BE}, there
is a convention to use equation (\ref{int03}) for 
$\Im z >0$, and to use equation (\ref{int02}) for 
$\Im z <0$. In fact, both equations are usually
combined into a single formula:
\begin{equation}\label{int04}
\eqalign{
\phi(a,c;z)
=\frac{\Gamma(c)}{\Gamma(c-a)} e^{i\pi a{\epsilon}}\psi(a,c;z)\\
+\frac{\Gamma(c)}{\Gamma(a)} e^{-i\pi(c-a){\epsilon}}
\psi(c-a,c;e^{-i\pi {\epsilon}}z) e^{z}, \qquad \epsilon = 
\mbox{sign}\,{\Im z}.}
\end{equation}
The reason for this is to ensure that when $-\pi < \arg z < \pi $
the both $\psi$-functions in the right-hand side have 
their arguments within  the interval $(-3\pi/2, 3\pi/2)$
and hence the canonical asymptotics (\ref{cha}) hold
for both terms. It is worth noticing, however, that there
is no prohibition for using (\ref{int04}), say with $\epsilon =1$
but for $\Im z < 0$. 

It is also worth mentioning one more time that, similar to equations
(\ref{rep1}) and (\ref{rep2}),
equations (\ref{int02}), (\ref{int03}), and (\ref{int04}) are valid
for all complex values of the parameters $a$ and $c \neq -n$,
$n =0, 1, 2, \dots$ 
\end{remark}

\end{document}